\documentclass[10pt,fullpage]{article}

\usepackage{amssymb}
\usepackage{amsmath}
\usepackage{dsfont}
\usepackage{amsthm}
\usepackage{epsfig}
\usepackage{enumerate}
\usepackage{multicol}
\usepackage{caption}
\usepackage{psfrag}
\usepackage{xcolor}
\usepackage{graphicx}
\usepackage{tcolorbox}
\usepackage{authblk}
\usepackage{fullpage}
\usepackage{showlabels}
\usepackage{todonotes}
\usepackage[normalem]{ulem} 

\input epsf

\numberwithin{equation}{section}

\newtheorem{hyps}{Hypotheses}

\newtheorem{Lemma}{Lemma}[section]

\newtheorem{Th}[Lemma]{Theorem}
\newtheorem{Prop}[Lemma]{Proposition}
\newtheorem{Def}[Lemma]{Definition}
\newtheorem{claim}{Claim}
\theoremstyle{remark}
\newtheorem{Rk}[Lemma]{Remark}

\newcommand{\be}{\begin{equation}}
\newcommand{\ee}{\end{equation}}
\newcommand{\baa}{\begin{array}}
\newcommand{\eaa}{\end{array}}
\newcommand{\ba}{\begin{eqnarray}}
\newcommand{\ea}{\end{eqnarray}}

\def\epsilon{\varepsilon}
\newcommand\redsout{\bgroup\markoverwith{\textcolor{red}{\rule[0.5ex]{2pt}{0.4pt}}}\ULon}

\def\R{{\mathbb R}}

\newcommand{\bee}{\begin{equation*}}
\newcommand{\eee}{\end{equation*}}
\newcommand{\bc}{\begin{cases}}
\newcommand{\ec}{\end{cases}}

\date{}
\title{Spreading and vanishing for a monostable reaction-diffusion equation with forced speed}
\author{Juliette Bouhours\footnote{CMAP, Ecole Polytechnique, Palaiseau, France} , Thomas Giletti\footnote{IECL, Universit\'{e} de Lorraine, Vandoeuvre-l\`{e}s-Nancy, France}}
\begin{document}
\maketitle
\begin{abstract}
Invasion phenomena for heterogeneous reaction-diffusion equations are contemporary and challenging questions in applied mathematics. In this paper we are interested in the question of spreading for a reaction-diffusion equation when the subdomain where the reaction term is positive is shifting/contracting at a given speed $c$. This problem arises in particular in the modelling of the impact of climate change on population dynamics. By placing ourselves in the appropriate moving frame, this leads us to consider a reaction-diffusion-advection equation with a heterogeneous in space reaction term,  in dimension $N\geq1$. We investigate the behaviour of the solution $u$ depending on the value of the advection constant~$c$, which typically stands for the velocity of climate change. We find that, when the initial datum is compactly supported, there exists precisely three ranges for $c$ leading to drastically different situations. In the lower speed range the solution always spreads, while in the upper range it always vanishes. More surprisingly, we find that that both spreading and vanishing may occur in an intermediate speed range. The threshold between those two outcomes is always sharp, both with respect to $c$ and to the initial condition. We also briefly consider the case of an exponentially decreasing initial condition, where we relate the decreasing rate of the initial condition with the range of values of~$c$ such that spreading occurs.\medskip

\end{abstract}

\section{Introduction}

In this paper we consider the following reaction-advection-diffusion problem, motivated by the modelling of climate change:
\be\label{problemmoving}
\partial_t u-\Delta u-c\partial_{x} u=f(x,y,u), \quad t>0, \ (x,y)\in \Omega = \R\times\omega, 
\ee
where $\omega \subseteq \mathbb{R}^{N-1}$ with $N \geq 1$, and $c \geq 0$. Our purpose is the mathematical study of the large time behaviour of solutions of this problem, and in particular to determine whether the solution goes to 0 or not as $ t \to +\infty$.

\subsection{Motivation}
By the simple change of variable $u(t,x,y) = \tilde{u} (t,x+ct,y)$, this problem is equivalent to the following:
\be\label{probleminit}
\partial_t \tilde{u} - \Delta \tilde{u} = f (x-ct,y,\tilde{u}), \quad t>0, \ (x,y)\in \R\times\omega.
\ee
This latter equation is motivated by the study of the effect of climate change on population persistence. In~\cite{MM00,WPetal02}, the different authors point out that global warming induces a shift of the climate envelope of some species toward higher altitude or latitude. Therefore, these species need to keep pace with their moving favourable habitat in order to survive. In such a framework, $u$ stands for the density of an ecological population, which spatially disperses according to a diffusion process, and the reaction term $f$ is its growth rate. The dependence of $f$ on the moving variable $x-ct$ means that the environment is moving in the $x$-direction, due to the climate change phenomenon whose speed is the parameter~$c$. It is typically assumed that $f$ is negative in the region $-x +ct>>1$, so that the population habitat drifts to the right. We also refer to~$c$ as the forced speed because the population is forced to move at least with speed~$c$ to the right in order to survive.

More precisely, while the growth term is negative in the unfavourable zone far to the left, we will assume that it is increasing with respect to the variable $x-ct$ and that it converges to a monostable growth term far to the right (as it will be stated below). We already highlight here that by monostable, we only mean (see Definition~\ref{def:mono}) that the growth rate of a population is positive when $u \in (0,1)$ (after renormalization, the constant~1 stands for the carrying capacity in the climate envelope). However, we do not assume that the per capita growth rate is maximal at $u=0$, that is the population may be subjected to a weak Allee effect, which is a common feature of ecological species arising for instance from cooperative behaviours such as defense against predation~\cite{allee}.

Such shifting range models were introduced in several papers to study similar ecological problems. In \cite{PL} and \cite{BDD}, the effect of climate change and shifting habitat on invasiveness properties is studied for a system of reaction-diffusion equations. In~\cite{BDNZ,BR1,BR2,Lietal14,Vo}, the authors deal with persistence properties under a shifting climate for a scalar reaction-diffusion equation in dimension 1 and higher, and exhibit a critical threshold for the shifting speed below which the species survives and above which it goes extinct. This problem has also been studied in the framework of integrodifference equations, where time is assumed to be discrete and dispersion is nonlocal~\cite{HZetal14,ZK,ZK13}. However, all the aforementioned papers hold under KPP type assumptions whose ecological meaning is that there is no Allee effect and which mathematically imply that the behaviour of solutions is dictated by a linearized problem around the invaded unstable state. In the context of~\eqref{probleminit}, the KPP assumption typically writes as $f(z,s) \leq \partial_s f(z,0) s$ for all $z \in \R$ and $s \geq 0$.

On the other hand a few papers investigate the case when no KPP assumption is made. In \cite{BRRK}, the authors analyse numerically the effect of climate change and the geometry of the habitat in dimension 2, again in the framework of reaction-diffusion equations, with or without Allee effect. In \cite{BN}, Bouhours and Nadin consider~\eqref{probleminit} when the size of the favourable zone is bounded under rather general assumptions on the reaction~$f$ in the favourable zone (including the classical monostable and bistable cases). More precisely, they have shown the existence of two speeds $\underline{c} < \overline{c}$ such that the population persists for large enough initial data when $c < \underline{c}$ and goes extinct when $c > \overline{c}$. However, for a fixed initial datum, it is not known in general whether there exists a sharp threshold speed delimiting persistence and extinction. We will prove here that, when~$f$ satisfies the Hypotheses~\ref{hypsPb}, that is when the favorable habitat is contracting but its size is not bounded, the answer is positive but also that, unlike in the KPP case~\cite{BDNZ,BR1,BR2,Lietal14,Vo}, the threshold speed depends non trivially on the initial datum.\\

In this paper we are interested in the asymptotic behaviour of the solution $u(t,x,y)$ of \eqref{probleminit} as time goes to infinity, and in particular whether the solution goes extinct or spreads, in the sense of Definition~\ref{def:dyn} below. To do so we will study the problem in the moving frame and consider $u$ to be the solution of Problem~\eqref{problemmoving}, where now $x$ stands for the moving variable and $u$ for the density of the population in the moving environment.

\subsection{Assumptions}

As far as the cross section $\omega$ is concerned, we will focus on the two particular cases 
$$\omega = \mathbb{R}^{N-1}$$
or
$$\omega \mbox{ is an open, smooth and bounded subset of } \R^{N-1}.$$
In this latter case, equation~\eqref{problemmoving} is supplemented with a Neumann boundary condition
$$\partial_n u = 0 , \quad t>0, \ (x,y) \in \R \times \partial \omega,$$
where $n$ denotes the outward unit normal vector at the boundary.
However, up to some slight modifications, other types of boundary conditions such as Dirichlet or Robin boundary conditions may be considered, as discussed in Remark \ref{rk:bdrycond} below. 

Most of our results, and in particular the existence of a threshold speed, hold in both those cases. However, a situation where the cross section is bounded is more prone to mathematical analysis, which allows us to prove that the threshold between extinction or spreading of the solution is also sharp with respect to the initial data. Lastly, we recall that $N\geq 1$, and therefore the bounded cross section case also include the one dimensional case, i.e. $\omega = \{0\} =\mathbb{R}^0$.

\begin{Rk}\label{rk:bdrycond}
Our results remain valid up to relatively minor changes when considering other types of boundary conditions, such as Dirichlet or Robin. The main change is that the value of the linear (resp. nonlinear) speed $2 \sqrt{g_+ '(0)}$ (resp. $c^*$) introduced below should be modified accordingly with the boundary condition. In particular, both speeds may then depend on the shape of the spatial domain $\Omega$, and one may need to assume that the cross section $\omega$  is large enough in order to avoid a situation where extinction always occurs. For the sake of convenience, we restrict ourselves to the Neumann case.
\end{Rk}

Let us now turn to our assumptions on the reaction term~$f$:
\begin{hyps}\label{hypsPb}
\begin{itemize}
\item[$(1)$] The function $x \in\R\mapsto f(x,y,s)$ is nondecreasing for all $s\geq 0$ and $y \in \omega$.
\item[$(2)$] The function $s\mapsto f(x,y,s)$ is $C^{1,r}$ where $0<r<1$, and its $C^{1,r}$ norm on any bounded interval is uniformly bounded with respect to $x\in\R$ and $y\in\omega$. 
\item[$(3)$] There exist $C^{1,r} (\mathbb{R},\mathbb{R})$ functions $g_- $ and $g_+$ such that
$$\lim_{x \to \pm \infty} f (x,y,s) = g_\pm (s) \quad , \quad \lim_{x \to \pm \infty} \partial_s f (x,y,s) = g_\pm ' (s),$$
where the convergences hold in the $C^{0,r} (\mathbb{R}, \mathbb{R})$-topology, uniformly with respect to~$y$. Moreover
\begin{itemize}
\item[$(3')$] $g_- (0) = 0$ and $g_- (s) \leq - s$ for all $s \geq 0$; 
\item[$(3'')$] $g_+$ is of the monostable type, in the sense of Definition~\ref{def:mono} below.
\end{itemize}
\end{itemize}
\end{hyps} 
\begin{Rk}
One may in fact weaken our regularity assumptions on~$f$. In particular, the $s$-derivative of $f$ needs only be H\"{o}lder in a neighborhood of~0. We state our assumptions in such a way for simplicity. 
\end{Rk}
In our proofs, we will make an extensive use of comparison arguments and sliding methods, which can be performed thanks to part~$(1)$ of Hypotheses~\ref{hypsPb}. We complete part~$(3)$ by specifying exactly what we mean throughout this paper by monostable.
\begin{Def}\label{def:mono}
We say that the $C^{1,r}$ function $g_+$ is monostable if
$$g_+ (s) >0 \mbox{ if } s \in (0,1) \mbox{ and } g_+ (s) < 0 \mbox{ if } s \in (-\infty,0) \cup (1, + \infty),$$
$$g_+ ' (0) > 0 > g_+ '(1).$$	
\end{Def}
Note that this definition includes but is not restricted to the so-called Fisher-KPP case, where also $g_+ (u) \leq g_+ ' (0) u$ for all $u \geq 0$.

Under the above assumptions, it is clear that $f(x,y,0)= 0$ for all $x \in \mathbb{R}$ and $y\in\omega$, hence $0$ is a trivial steady state of~\eqref{problemmoving}. In a similar fashion, it is clear that any constant $M \geq 1$ is a supersolution of~\eqref{problemmoving}, which means that in large time the density $u$ should not exceed~1, the constant 1 being the stable state of the limiting equation far to the right.\medskip

Lastly, we add the initial conditon 
\be\label{initcond}u(0,x,y)=u_0(x,y) \ \mbox{ where } \
 0 \leq \not \equiv u_0 \in L^\infty (\R) \mbox{ and } support (u_0) \subset (-\infty, X] \times \omega ,\ee
for some $X>0$. We will also briefly discuss some other types of initial data which decay exponentially or slowly to the right. From a modelling point of view, it may indeed seem natural to assume that the initial population is already established in the whole favourable zone, which in our framework means the right half of the spatial domain. However, we chose to focus on the case when the initial support, or in other words the initial habitat, is bounded to the right. This means that we study spatial invasions into unoccupied regions and thus we mostly deal with invasive species which are initially introduced in a new region, but this assumption also accounts for native species whose distribution are not at equilibrium with current climate~\cite{araujo}. Moreover, this is when our model exhibits the most interesting phenomena, and in particular our analysis highlights the determining role of the initial size of the population. This is to be expected in an ecological setting, as well as in the mathematical model when the favourable zone is bounded. We refer again to~\cite{BN} for partial results as well as remaining open problems in such a framework.\\

Before stating our main results, we recall some properties of the usual monostable homogeneous equation
\be\label{pb:homomono}\begin{cases}
\partial_t v-\Delta v=g_+ (v), \quad t>0, \ z\in\R^N,\\
v(0,z)=v_0(z) \text{ compactly supported, nonnegative and nontrivial}.
\end{cases}\ee
It is well-known that the parabolic equation admits planar travelling wave solutions, i.e positive solutions of the form $v(t,z)=V(z\cdot\nu-ct)$ where $\nu$ is an arbitrary unit vector in $\R^N$, and $V$ satisfies the limits $V(-\infty) = 1$ and $V(+\infty) = 0$. More precisely, such a travelling wave exists if and only if $c \geq c^*$, where $c^* \geq 2 \sqrt{g_+ '(0)}$. In particular, we often refer to $c^*$ as the minimal speed. Moreover, $c^*$ is the asymptotic speed of propagation of solutions of Problem~\eqref{pb:homomono}.

Under the additional KPP hypothesis, i.e. $g_+ (u) \leq g_+ '(0) u$ for all $u \geq 0$, then $c^* = 2 \sqrt{g_+'(0)}$ which corresponds to the minimal speed of the linearized problem at~0. However, as mentioned above, we consider the general monostable case so that the nonlinear speed $c^*$ may be strictly larger than $2 \sqrt{g_+ '(0)}$. Our main results will show that both the linear and nonlinear speed play an important role in the behaviour of solutions of~\eqref{problemmoving}.

\subsection{Main results}
Let us now turn to the study of the large time behaviour of $u$, the solution of Problem \eqref{problemmoving}. Under our assumptions and because we assume little regularity with respect to space variables in the reaction term, solutions are to be understood in a weak sense. Still, solutions satisfy local $W^{1,2}_{p}$ ($p>1$) estimates which is enough to perform comparison principles. In particular the well-posedness as well as the nonnegativity and boundedness of solutions for all positive times follows.

We will exhibit different possible behaviours, which we define below respectively as spreading, vanishing and grounding.
\begin{Def}\label{def:dyn}
Let $u(t,x,y)$ be the solution of \eqref{problemmoving}. We say that:
\begin{itemize}
\item \textbf{spreading} occurs if $\lim_{t \to +\infty} u(t,\cdot) \to p_c^+$ locally uniformly in space, where $p^+_c (x,y)$ is a positive and bounded steady state which tends to 0 (resp. 1) as $x \to -\infty$ (resp. $x \to +\infty$) uniformly with respect to $y$;
\item \textbf{vanishing}, or \textbf{extinction}, occurs if $\lim_{t \to +\infty} u(t,\cdot) \to 0$ uniformly in space;
\item \textbf{grounding} occurs if for any $(x,y) \in \Omega$, we have $\displaystyle{\liminf_{t\to+\infty} u(t,x,y) >0}$, but for any $\varepsilon >0$ there exists $X_\varepsilon >0$ such that
$$\limsup_{t \to +\infty} \sup_{|x| \geq X_\varepsilon} u(t,x,y) \leq \varepsilon.$$
\end{itemize}
\end{Def}
We will show in Section~\ref{sec:classStat} that for any $c \geq 0$ there exists a unique such steady state $p^+_c$, and moreover it is the maximal positive and bounded steady state; see Propositions~\ref{p+c_exists} and~\ref{otherstates}. Note that, from the asymptotics of $p^+_c$ as $x \to \pm \infty$, it may be more accurate to say that the solution only spreads in the favourable zone: far into the unfavourable zone the solution actually decreases to 0, as one should naturally expect.

Let us also mention that, in the grounding case, the number $X_\varepsilon$ only depends on $c$ and $\varepsilon$ but not on the specific choice of the initial datum. Biologically, this means that the number $X_\varepsilon$ would act as un upper bound on the width of the habitat. However, the main point of our results is that the grounding phenomenon is highly unstable, and therefore we will not go into such details.\\

We now proceed to the statement of our main results. Our first theorem exhibits three speed ranges leading to quite different situations:
\begin{Th}\label{th:regime}
Assume that Hypotheses~\ref{hypsPb} hold, and let $u(t,x,y)$ be the solution of~\eqref{problemmoving} together with the initial condition~\eqref{initcond}. 
\begin{enumerate}[$(i)$]
\item If $0 \leq c < 2 \sqrt{g'_+(0)}$, then spreading occurs.
\item If $2\sqrt{g'_+(0)} \leq c < c^*$ (provided such $c$ exists), then both spreading and vanishing may occur depending on the choice of $u_0$. To be more precise, there exist initial data $u_{0,1} > u_{0,2}$ as in~\eqref{initcond} such that spreading occurs for $u_{0,1}$, and vanishing occurs for $u_{0,2}$.
\item If $c \geq c^*$, then vanishing occurs.
\end{enumerate}
\end{Th}
Recall from~\cite{AW} that the minimal wave speed of~\eqref{pb:homomono} is also the spreading speed of solutions of the Cauchy problem with nonnegative compactly supported initial data. Therefore, the third statement of Theorem~\ref{th:regime} simply means that, when the climate shifts faster than the species spreads in a favourable environment, then the species cannot keep pace with its climate envelope and goes extinct as time goes to infinity. On the other hand, when $c$ is less than $2 \sqrt{g'(0)}$ (statement $(i)$ of Theorem~\ref{th:regime}) which is the speed associated with the linearized problem around $u=0$, then any small population is able to follow its habitat and ultimately thrive. In particular, when $c^* = 2 \sqrt{g'(0)}$, we retrieve a sharp threshold speed between persistence and extinction, which as in the KPP framework of~\cite{BDNZ,BR1,BR2,Lietal14,Vo} does not depend on the initial datum. Let us note here that, while the KPP assumption implies that $c^* = 2 \sqrt{g'(0)}$, the converse does not hold~\cite{hrothe}.

Nonetheless, a striking feature of Theorem~\ref{th:regime} is the fact that when $c\in [2\sqrt{g'(0)}, c^*)$, whether the solution persists or not depends on the initial datum; see statement $(ii)$ above. This is in sharp contrast with the so-called `hair-trigger effect' for the classical homogeneous monostable equation~\eqref{pb:homomono}, whose solution spreads as soon as the initial datum is nontrivial and nonnegative. This result also highlights qualitative differences with the KPP framework of~\cite{BDNZ,BR1,BR2,Lietal14,Vo}, where the persistence of the population depends only on the value of $c$. The biological implication is that under the combination of a weak Allee effect and a shifting climate, the size and the position of the initial population become crucial for the survival of the species.\\

Our two next results aim at showing that the transition from spreading to vanishing is always sharp. First we fix the initial datum and vary the parameter~$c$:
\begin{Th}\label{th:threshold_speed}Assume that Hypotheses~\ref{hypsPb} hold, and let $u(t,x,y)$ be the solution of~\eqref{problemmoving} together with the initial condition~\eqref{initcond}. 

\begin{enumerate}[$(i)$]
\item There exists a threshold speed $c (u_0) \in [2 \sqrt{g_+ '(0)},c^*]$ such that, if $c < c (u_0)$ spreading occurs and if $c > c(u_0)$ vanishing occurs (in the sense of Definition \ref{def:dyn} above).
\item If moreover one of the following two conditions holds:
\begin{enumerate}[$(a)$]
 \item $\omega$ is bounded (including the case $N=1$);
\item $f(x,s)$ does not depend on $y$ and $support (u_0) \subset (-\infty, X] \times [-Y,Y]^{N-1}$ for some $X >0$ and $Y>0$;
\end{enumerate}
then at the threshold speed $c = c(u_0)$:
\begin{itemize}
\item either $c (u_0) \in  (2 \sqrt{g_+ '(0)}, c^*)$ and grounding occurs;
\item or $c  (u_0) = 2 \sqrt{g_+ '(0)}$ and then both vanishing and grounding may occur.
\end{itemize}
In particular, the situation $c (u_0) = c^* > 2 \sqrt{g_+ '(0)}$ never occurs.
\end{enumerate}
\end{Th}
The main point of Theorem~\ref{th:threshold_speed} is its first statement, which shows that there still exists, in the general monostable framework, a threshold forced speed below which spreading occurs and above which the solution goes extinct. As mentioned above, the existence of such a threshold for persistence was already known in the KPP framework. However, Theorem~\ref{th:threshold_speed} together with Theorem~\ref{th:regime} clearly imply that $c (u_0)$ depends in a nontrivial way on the initial datum as soon as $c^* > 2 \sqrt{g'(0)}$. 
From an ecological point of view, this means that the persistence of the population is determined by the value of the climate shift speed with respect to this threshold. When the per capita growth rate of the population is optimal at zero density (no Allee effect), this threshold speed is independent of the initial datum, whereas in the presence of a weak Allee effect, this threshold depends on the size of the initial datum. Furthermore, statement $(ii)$ of Theorem~\ref{th:threshold_speed} also provides a more precise picture of the behaviour of the solution in the threshold case. While our results suggest that grounding should not be biologically observed (this is in fact related to our assumption that the favourable zone is unbounded), this notion will play an important role in the proofs of both statements~$(i)$ and~$(ii)$ of Theorem~\ref{th:threshold_speed}.

In our second sharp threshold type result, we instead fix the speed parameter~$c$ and consider a family of initial data. We then investigate the sharpness of the threshold from spreading to vanishing behaviours, in the same spirit as the results of~\cite{DM,DP,MZ,polacik,zlatos} in the spatially homogeneous framework. In our situation, this issue is rather complicated because we cannot use symmetrization techniques as in~\cite{DP,polacik} in order to deal with the case of a multidimensional domain with an unbounded cross section~$\omega$. This leads us to restrict ourselves to a bounded cross section $\omega$, including the one dimensional case.
\begin{Th}\label{th:threshold_u0}
Let Hypotheses~\ref{hypsPb} be satisfied, and assume also that $\omega$ is bounded and $2 \sqrt{g_+ '(0)} < c^*$. Consider $c \in [2\sqrt{g_+ '(0)},c^*)$ and a continuous (in the $L^1 (\R \times \omega)$-topology) family $(u_{0,\sigma})_{\sigma >0}$ of nonnegative, compactly supported and continous initial data, ordered in the following way:
$$\forall \sigma_1 < \sigma_2, \quad  
 u_{0,\sigma_1} < u_{0,\sigma_2} \mbox{ in the support of $u_{0,\sigma_1}$}.$$
Then there exists $\sigma_* \in [0,+\infty]$ such that the corresponding solution $u_\sigma (t,x,y)$ spreads if $\sigma > \sigma_*$, vanishes if $\sigma < \sigma_*$, and is grounded for $\sigma = \sigma_*$ (whenever $\sigma_* \in (0,+\infty)$). 
\end{Th}
Theorem~\ref{th:threshold_u0} again highlights how the outcome depends on the initial datum. It shows that extinction and spreading are the two reasonable outcomes, while grounding is a rather unstable phenomenon. Indeed, while there exist initial data such that grounding occurs, any mild increase (resp. decrease) of such a datum will lead to spreading (resp. extinction) of the solution. In particular, a larger initial population is more likely to survive despite the shift of the favourable zone. Note that, recalling that the support of a function is the closure of the set where it is positive, the assumptions of Theorem~\ref{th:threshold_u0} also imply that the support of $u_{0,\sigma}$ strictly increases when $\sigma$ increases. The proof that we use relies on this hypothesis but we believe that it is not necessary and that it is enough to assume that $u_{0,\sigma_2} \geq \not \equiv u_{0,\sigma_1}$.\\

All our results above deal with initial data whose support is bounded in the right direction, see~\eqref{initcond}. As we mentioned before, the biological interpretation of this hypothesis is that the initial population is not at equilibrium with climate. On the other hand, if the initial population already inhabits the whole favourable zone, i.e. it does not decay to the right, then the situation is rather different and the population always persists, regardless of the climate change speed~$c$ and the specific choice of~$u_0$.

This leads us to consider the intermediate situation of exponentially decaying initial data. In our last result, we are able to find precisely when the survival of the species depends on the size of the initial datum, and when it depends only on its decay as $x \to +\infty$.
\begin{Th}\label{th:last}
Let Hypotheses~\ref{hypsPb} be satisfied and define $\alpha^* := \frac{c^*- \sqrt{ c^* \,^2 - 4 g_+ '(0)}}{2}.$ 
Let also $u(t,x,y)$ be the solution of~\eqref{problemmoving} with an initial condition $u_0$ which is bounded, nonnegative, and such that
$$u_0 (x,y) \sim A e^{-\alpha x}$$
as $x \to +\infty$ where $\alpha >0 $ and $A>0$. Define also
$$c_\alpha := \frac{\alpha^2 + g'_+ (0)}{\alpha} \geq 2 \sqrt{g'_+ (0)}.$$
\begin{enumerate}[$(i)$]
\item If $\alpha \geq \sqrt{g'_+ (0)}$, then spreading occurs for all $c < 2 \sqrt{g'_+ (0)}$ and vanishing occurs for all $c > c^*$. In the speed range $(2 \sqrt{g'_+ (0)},c^*)$, both spreading and vanishing may occur depending on the initial datum~$u_0$.
\item If $\alpha \in (\alpha^*, \sqrt{g'_+ (0)})$, then $2 \sqrt{g'_+ (0)} < c_\alpha < c^*$, spreading occurs for all $c < c_\alpha$ and vanishing occurs for all $c > c^*$. In the speed range $(c_\alpha,c^*)$, both spreading and vanishing may occur depending on the initial datum~$u_0$.
\item If $\alpha \leq \alpha^*$, then $c_\alpha \geq c^*$, spreading occurs for all $c < c_\alpha$ and vanishing occurs for all $c > c_\alpha$.
\end{enumerate}
\end{Th}
Theorem~\ref{th:last} shows how the dynamics remarkably changes with the decay rate parameter~$\alpha$. First, when~$\alpha$ is large enough, the situation is the same as in Theorem~\ref{th:regime}. However, as one reduces $\alpha$, the intermediate speed range where both spreading and vanishing occur starts shrinking, and eventually disappears at $\alpha = \alpha^*$. Then, when $\alpha < \alpha^*$, the threshold speed~$c_\alpha$ no longer depends on the specific choice of the initial datum. Finally, the threshold speed~$c_\alpha$ clearly tends to $+\infty$ as $\alpha \to 0$, which implies that spreading always occurs when the initial datum does not decay, or decays slower than any exponential, as $x \to +\infty$.

Naturally, the fact that the outcome is so sensitive to the decay of the initial datum is related to the unboundedness of the favourable zone as well as our choice of a monostable nonlinearity. Though our arguments could be extended to the bistable (strong Allee effect) case, handling a bounded favourable zone is in fact much more complicated. Indeed, while some of our results remain valid, in particular the existence of an intermediate range speed where the outcome depends on the initial datum, some new phenomena are also expected due to the lack of monotonicity with respect to the climate shift speed. We hope to investigate this issue in a future work.

\bigskip

\textbf{\underline{Organisation of the paper}}
\bigskip

In Section \ref{sec:classStat} we study the stationary problem corresponding to~\eqref{problemmoving} and prove that there exists a (unique) maximal stationary solution that does not decay to 0 as $x\to+\infty$. In Section~\ref{sec:3speed} we examine the different speed ranges and prove Theorem~\ref{th:regime}. In particular we show that the hair-trigger effect does not hold in the intermediate speed range $c \in [2 \sqrt{g'(0)},c^*)$. In Section \ref{sec:sharpthresholds}, we investigate the existence of sharp threshold phenomena either with respect to the speed $c$ or with respect to the initial datum $u_0$. A key step is some uniform in time exponential estimates on any solution with a compactly supported initial datum, see Proposition~\ref{th:boundexp}. Then Section~\ref{sec:sharpspeed} and Section~\ref{sec:sharpinit} are devoted respectively to the proofs of Theorems~\ref{th:threshold_speed} and~\ref{th:threshold_u0}.  Our last Theorem~\ref{th:last} on the behaviour of solutions of~\eqref{problemmoving} with exponentially decaying initial conditions is dealt with in Section~\ref{sec:last}. Finally, we include in an Appendix~\ref{A:maxpple} two maximum principles (elliptic and parabolic) that will be used throughout different proofs in the paper.

\section{Classification of the stationary solutions}\label{sec:classStat}

This section is devoted to stationary solutions of \eqref{problemmoving}. First, we prove the existence of a maximal positive and bounded steady state~$p_c^+$ which is involved in the spreading notion that we introduced in Definition~\ref{def:dyn}; see Proposition~\ref{p+c_exists}. Then we will show that~$p_c^+$ is the unique (positive and bounded) stationary solution which does not decay as $x \to +\infty$. Among the stationary solutions which decay as $x \to +\infty$, some have a faster exponential decay (we will refer to them as ground states) and will play an important role in the proof of our sharp threshold results. Proposition~\ref{othercritical} states that, at least when $\omega$ is bounded, the set of ground states cannot be ordered. This will in turn guarantee, in a later section, that grounding phenomena (in the sense of Definition~\ref{def:dyn}) are unstable with respect to perturbations of the initial datum.

\subsection{The maximal steady state}

The first property we prove is the existence of a maximal positive stationary solution $p_c^+$. We recall that large time convergence to this maximal steady state is referred to as spreading.
\begin{Prop}\label{p+c_exists}
For any $c \geq 0$, there exists a positive steady state $p^+_c (x,y)$ such that $p^+_c(-\infty,\cdot)=0$, $p^+_c(+\infty,\cdot)=1$ (both convergences are understood to be uniform with respect to the second variable).

Moreover, it is increasing with respect to~$x$, and it is maximal among positive and bounded steady states, i.e. any other such steady state $p$ satisfy $p < p^+_c$.

\end{Prop}
\begin{proof}
We first construct $p^+_c$.  
Choose any constant $M >1$, and note that $M$ is a supersolution of \eqref{problemmoving}. Recall also that 0 is a trivial solution. Thus the solution~$u^M$ of~\eqref{problemmoving} with initial condition $u_0\equiv M$ is decreasing in time and, by standard parabolic estimates, it converges locally uniformly to a nonnegative and bounded stationary solution. Denote this stationary solution by $p^+_c$ and prove that it satisfies the required properties.

\begin{claim}\label{claim1}
The function 
$x\mapsto p^+_c(x,\cdot)$ is increasing and $p^+_c(-\infty,\cdot)=0$ and $p^+_c(+\infty,\cdot)=1$, uniformly with respect to the second variable.
\end{claim}
First, by the monotonicity of $f(x,y,s)$ with respect to $x$, it is clear that for any $x' >0$ the function $u^M (\cdot + x', \cdot,\cdot)$ is a supersolution of~\eqref{problemmoving} which coincides with $u^M$ at time $t=0$. Therefore, by the comparison principle and passing to the limit as $t \to +\infty$, we conclude that $p_c^+ (x,y)$ is nondecreasing with respect to~$x$.

It now follows from the monotonicity and the boundedness of $p^+_c$ that the limits $p^+_c (\pm \infty,y)$ exist and are nonnegative. By standard estimates and passing to the limit in the equation, we get that these limits are at least locally uniform with respect to~$y$, and that 
$$\Delta p^+_c (-\infty,\cdot) + g_- (p^+_c (-\infty ,\cdot)) = 0,$$
as well as 
$$\Delta p^+_c (+\infty,\cdot) + g_+ (p^+_c (+\infty ,\cdot)) = 0.$$
On one hand, it easily follows from the first equation and the negativity of $g_-$ that $p^+_c (-\infty,\cdot) =0$. In the case when $\omega = \mathbb{R}^{N-1}$, it remains to check that the convergence is uniform with respect to~$y$, but the argument is in fact very similar. Indeed, for any sequences $x_n \to - \infty$ and $y_n\in\R^{N-1}$ we consider $p^+_n (\cdot)=p^+_c(x+x_n,y+y_n)$ and using standard elliptic estimates we can conclude that $p^+_n$ converges to $p^+_\infty$ as $n\to+\infty$, with $p^+_\infty$ a bounded solution of $\Delta p^+_\infty+c\partial_x p^+_\infty+g_-(p^+_\infty)=0$, with $g_-(s)\leq -s$ for all $s\in\R^+$. This again implies that $p^+_\infty \equiv 0$, and because of our arbitrary choice of sequences, the uniform convergence of $p^+_c (x,\cdot)$ to 0 as $x \to -\infty$ follows. 

On the other hand, the second equation is of the monostable type and admits two nonnegative stationary states, namely~0 and~1. Proceeding as above, we reach the conclusion that
\be\label{lem1ing1}
\limsup_{x \to +\infty} \sup_{y \in \omega} p^+_c (x,y) \leq 1,
\ee
which in particular, together with a strong maximum principle and the monotonicity of $p^+_c$ with respect to $x \in \mathbb{R}$, implies that
$$p^+_c < 1.$$
However we need to rule out the possibility that $p^+_c (+\infty, \cdot)$ may be~0. To do so, we go back to our construction of $p^+_c$ above and introduce some subsolutions, using the fact that the limit equation as $x \to +\infty$ admits travelling waves with positive speed. For any small $\delta >0$, one can find a function $g_\delta$ and a positive real number $X_\delta$ such that, for all $x \geq X_\delta$, $y\in\omega$ and $s \in [0,M]$, we have
$$f (x,y,s) \geq g_\delta (s),$$
and moreover $g_\delta$ is an ignition type function, namely $g_\delta (s) = 0$ if $s \in [-\delta,0] \cup \{1-\delta\}$, and $g_\delta (s) >0$ if $s \in (0,1-\delta)$. This is another standard case (see~\cite{AW}) in which the corresponding reaction-diffusion equation
$$\partial_t v - \Delta v = g_\delta (v),$$
which here we pose in the spatial domain $\mathbb{R} \times \omega$ (with a Neumann boundary condition if $\omega$ is bounded) again admits travelling wave solutions. More precisely, there exists $c_\delta >0$ and a global in time solution $V_* (x-c_\delta t)$, satisfying the limits $V_* (-\infty) = 1 - \delta$ and $V_* (+\infty) = - \delta$. Rewriting the above equation in the moving frame with speed $c_\delta$ in the $x$-direction, one gets as usual the following ordinary differential equation:
$$V_* '' + c_\delta V_* ' + g_\delta (V_*) = 0.$$
Furthermore, it is known that $V_*$ is a decreasing function~\cite{AW}. Up to some shift, we can then assume that
$$V_* (0) = 0 \ \mbox{ and } V_* < 0 \mbox{ in } (0,+\infty).$$
Let us now introduce
$$\underline{u} (x) := \max \{ 0 , V_* (-x+ X_\delta ) \} \},$$
and show that this is a subsolution of \eqref{problemmoving}. To check this claim, we only need to consider the half-space $\{ x  > X_\delta \}$ where $\underline{u} >0$. We compute
\begin{eqnarray*}
\partial_t \underline{u} - \partial_{x}^2 \underline{u} - c \partial_{x} \underline{u} - f(x,y,  \underline{u}) & \leq & - V_* '' + c V_* ' - g_\delta (V_*)\\
& \leq & (c + c_\delta) V_* ' \\
& \leq & 0,
\end{eqnarray*}
where the last inequality follows from the monotonicity of $V_*$. We conclude that we have indeeed constructed a subsolution $\underline{u} \leq M$ and, by comparison, we get that $\underline{u} \leq p^+_c$. In particular,
$$\liminf_{x \to +\infty} \inf_{y \in \omega} p^+_c (x,y) \geq  V_* (-\infty) = 1 -\delta.$$
As $\delta$ could be chosen arbitrarily small and recalling also~\eqref{lem1ing1}, this proves that $p^+_c (+\infty,y)=1,$
uniformly with respect to~$y$.

To complete the proof of Claim~\ref{claim1}, it only remains to see that $p^+_c$ is not only nondecreasing but also increasing with respect to~$x$, which simply follows from the strong maximum principle. The next claim concludes the proof of Proposition~\ref{p+c_exists}.

\begin{claim}\label{claim2}
The steady state~$p^+_c$ is maximal among positive and bounded steady states, i.e. any other such steady state $p$ satisfy $p < p^+_c$ in $\Omega$.
\end{claim}
Before proving this second claim, let us first observe that the construction of $p^+_c$ above actually did not depend on the choice of the constant $M>1$. Indeed, choose two constants $M_1 > M_2 > 1$ and let $p_1$, $p_2$ be the limits as $t \to +\infty$ of the corresponding solutions. The comparison principle immediately implies that $p_1 \geq p_2$. But also, $M_2 > 1 \geq p_1$ and hence $p_2 \geq p_1$. Therefore $p_1 \equiv p_2$ and, as announced, the state $p^+_c$ constructed above does not depend on the choice of~$M>1$.

Claim~\ref{claim2} easily follows. For any bounded steady state $p$, we can choose~$M$ large enough so that $p \leq M$ and $p \leq p^+_c$ by the comparison principle. Then, by the strong maximum principle, either $p \equiv p^+_c$ or $p < p^+_c$. This also ends the proof of Proposition~\ref{p+c_exists}.\end{proof}

\subsection{Stationary solutions and ground states}

	We now prove that $p^+_c$ is the unique positive and bounded state which converges to 1 as $x \to +\infty$, and other such states actually converge to 0. We will then also prove that, when $\omega$ is bounded, steady states which have a fast enough decay as $x \to +\infty$ (in a sense to be made precise below) cannot be ordered.
\begin{Prop}\label{otherstates}
Any other positive and bounded stationary solution~$p (x,y)$ converges to 0 as $x \to +\infty$ locally uniformly with respect to~$y$.
\end{Prop}
Note that this property implies that, if $f(x,s)$ does not depend on $y$, then neither does $p^+_c$, as the problem is ``independent of $y$'' and $p^+_c$ is the unique steady state that does not converges to 0 as $x\to+\infty$. In a similar fashion, if $f(x,y,s)$ is periodic with respect to $y$, then so is $p^+_c$. Those facts actually easily follow from the above construction, where $p^+_c$ was obtained as the limit of the Cauchy problem with a (large) constant initial datum. 

While Proposition~\ref{otherstates} will be enough for our purpose, one may still wonder if we can prove more than that. In fact, it is clear from our first main Theorem~\ref{th:regime} that, when $c < 2 \sqrt{g_+ '(0)}$, then there does not exist any positive and bounded steady state other than $p^+_c$. 
Note also that when $c < c^*$, by a combination of the argument below and the next section, the convergence to 0 in the above proposition is also uniform with respect to~$y$. Indeed, if the convergence to 0 is not uniform with respect to~$y$, then one can show that $p(x,y)$ converges locally uniformly to 1 along some sequence $(x_n,y_n)$ with $x_n \to +\infty$ and $|y_n| \to +\infty$. In particular, one can put below $p$ a compactly supported subsolution as in the spreading  scenario of Section~\ref{sec:interm}, which in turn implies that $p \equiv p^+_c$. However, when $c > c^*$, we expect that there exist conical wave solutions, which converge locally uniformly to 0 as $x \to +\infty$ but not uniformly; see for instance~\cite{hamelmonneau,ninotani} for details on conical waves.
\begin{proof}
Let us consider a stationary solution $p(x,y)$, positive and bounded, and assume also that it does not converge (locally uniformly in $y$) to 0 as $x \to +\infty$, i.e. there exists a sequence $(x_n,y_n)_n \subset \Omega$ with $(y_n)_n$ bounded and such that
$$\lim_{n \to +\infty} x_n = + \infty \quad , \quad \lim_{n \to +\infty} p(x_n,y_n) >0.$$
Up to extraction, we assume that $y_n \to y_\infty \in \overline{\omega}$ as $n \to +\infty$, and by standard elliptic estimates the function $p_n (x,y) = p_n (x + x_n, y)$ converges locally uniformly to a solution of 
$$\Delta p_\infty + c \partial_x p_\infty + g_+ (p_\infty) = 0,$$
in $\Omega$. 
Since $p_\infty (0,y_\infty) > 0$, we get by the strong maximum principle (together with Hopf lemma in the case when $\omega$ is bounded)
that $p_\infty >0$ in $\overline{\Omega}$. Therefore, for any $R >0$, there exists $\xi \in (0,1)$ and $n$ large such that
\begin{equation}\label{other1}
\forall n' \geq n, \ \forall (x,y) \in B_R (x_n,y_\infty) \cap \overline{\Omega}, \quad p(x,y) \geq \xi,
\end{equation}
where $B_R (x_n,y_\infty)$ denotes the ball of radius $R$ centered at $(x_n,y_\infty)$. 

Let us now construct a subsolution lying below the steady state~$p$. 
\begin{Rk}
We only write the details in the case $\omega = \R^{N-1}$. The case when $\omega$ is bounded is in fact simpler, as it is sufficient to construct a subsolution which depends only on $x$ and thus satisfies the Neumann boundary condition automatically.
\end{Rk}
Define $\Lambda_R$ and $\phi_R$ respectively the principal eigenvalue and principal eigenfunction of the Laplacian with Dirichlet boundary condition in $B_R (0)$ the ball of radius $R>0$ centered at 0, i.e.:
$$\left\{
\begin{array}{l}
\Delta \phi_R  = \Lambda_R \phi_R  \  \mbox{ in } B_R (0), \vspace{3pt}\\
\phi_R = 0 \ \mbox{ on } \partial B_R (0), \quad \phi_R >0 \  \mbox{ in } B_R (0).\\ 
\end{array}
\right.
$$
Up to some normalization, we assume that $\|\phi_R\|_\infty = 1$, and moreover we extend it by 0 outside $B_R (0)$. It is well-known that $\Lambda_R \to 0$ as $R \to +\infty$. We can let $R$ large enough so that
$$\Lambda_ R + \frac{g_+ ' (0) }{2} >0,$$
and then $\xi >0$ such that~\eqref{other1} holds. In particular $ p( \cdot) \geq \xi \phi_R (x-x_n,y-y_\infty)$ for any $n \geq n'$.

Up to decreasing $\xi$, there exists $X_\xi$ such that, for all $x \geq X_\xi$, $y \in \R^{N-1}$ and $s \in [0,\xi]$, we have
\be\label{machin11}
f(x,y,s) \geq \frac{g_+ ' (0)}{2} s .
\ee
One can check that $\xi \phi_R (x-x_n +ct ,y-y_\infty)$ is a subsolution of \eqref{problemmoving} up to the time $T_n = \frac{x_n - R - X_\xi}{c}$, which is simply the time at which the support of $\phi_R$ first intersect the zone where \eqref{machin11} may no longer hold. Note that $T_n >0$ provided that $n$ is large enough. 

Now take any $x > X_\xi  + R$ and $y \in B_{R/2} (y_\infty)$. Choose $n$ large enough so that $x_n >x$, and $t = (x_n-x)/c \in (0,T_n)$. Applying the comparison principle one can get that
$$p(x,y)\geq \xi \phi_R(x-x_n+ct,y-y_\infty) =\xi \phi_R (0,y-y_\infty).$$
In particular,
\begin{equation}\label{other2}
\inf_{x \geq X_\xi +R , |y-y_\infty| \leq \frac{R}{2}} p(x,y) >\inf_{|y-y_\infty| \leq \frac{R}{2}} \xi \phi_R(0,y-y_\infty)>0 .
\end{equation}
As mentioned above, the same approach can be used and leads to a similar inequality when $\omega$ is bounded. We are now in a position to make the following claim.

\begin{claim}\label{claim11} It holds true that
$$\lim_{x \to +\infty} p(x,y) = 1,$$
locally uniformly with respect to $y \in \omega$.\end{claim}
First, we know that $\limsup_{x \to +\infty} p (x,y) \leq 1$, as follows from $p \leq p^+_c$ by Proposition~\ref{p+c_exists}. Then let $x_n \to +\infty$, and as before up to extraction of a subsequence $p (x+ x_n,y)$ converges locally uniformly to a solution $p_\infty$ of
\begin{equation}\label{other_lim}
\Delta p_\infty + c \partial_x p_\infty + g_+ (p_\infty) = 0.
\end{equation}
However, now we know not only that $p_\infty$ is positive, but also from~\eqref{other2} that
$$\eta_\infty := \inf_{x \in \mathbb{R}, |y- y_\infty| \leq \frac{R}{2} } p_\infty (x,0) >0.$$
\begin{Rk}
In the case when $\omega$ is bounded, one may use the strong maximum principle to conclude that $\inf p_\infty >0$. By a simple comparison with solutions of the underlying ODE, it immediately follows that $p_\infty \geq 1$, hence $p_\infty \equiv 1$. 
\end{Rk}
Let us again focus on the case when $\omega = \R^{N-1}$. By the classical results of Aronson and Weinberger~\cite{AW}, the solution of
$$\partial_t v = \partial_y^2 v + g_+ (v) $$
with initial datum $\eta_\infty \chi_{|y-y_\infty| \leq \frac{R}{2}}$ converges locally uniformly to~1. Here $\chi$ denotes a characteristic function. As the function $v$ does not depend on~$x$, it also satisfies
$$\partial_tv-\Delta v-c\partial_x v=g_+ (v)$$
in $\R^N$, with $v(0,x,y)\leq p_\infty(x,y)$ and, by the parabolic comparison principle, we conclude again that~$p_\infty =1$. The above claim is proved.

\begin{claim}\label{claim12}
The convergence of $p$ to~1 as $x \to +\infty$ is also uniform with respect to~$y$. 
\end{claim} 
\begin{Rk}
Obviously this part can be skipped when $\omega$ is bounded. 
\end{Rk}
The argument is quite similar to the proof of~\eqref{other2}. Let $\delta >0$ be arbitrarily small, and $X_\delta$ be such that $x \geq X_\delta$ and $y \in \R^{N-1}$ imply that $f(x,y,s) \geq g_+^\delta (s)$ where $g_+^\delta$ is a monostable type function, whose zeros are 0 and $1-\delta$. Then there exists a subsolution $\underline{v}_R^\delta$ of
$$\Delta \underline{v}_R^\delta + g_+^\delta (\underline{v}_R^\delta) \geq 0$$
in $\mathbb{R}^N$, with compact support $B_R (0)$ and such that $$\max \underline{v}_R^\delta = \underline{v}_R^\delta (0,0) = 1- 2\delta.$$
We refer again to Aronson and Weinberger~\cite{AW} for such a construction and omit the details. Proceeding as above, for any $y_0 \in \R^{N-1}$, we define $\underline{w}(t,x,y):=\underline{v}_R^\delta(x-n+ct,y-y_0)$ and observe that $\underline{w}$ is a subsolution of \eqref{problemmoving} for all $t\in(0,\frac{n-R-X_\delta}{c})$, assuming $n$ large enough (depending a priori on $y_0$). Thanks to Claim~\ref{claim11} and by the comparison principle, we can compare $p$ with $\underline{w}$ for any $n$ large enough, on the interval of time~$[0, \frac{n-R-X_\delta}{c}]$. This leads to the inequality, for any $y_0 \in \mathbb{R}^{N-1}$,
\begin{equation}\label{plim1b}
\inf_{x \geq X_\delta +R}  p(x,y_0) \geq 1 - 2\delta.
\end{equation}
Noting that $X_\delta$ and $R$ do not depend on $y_0$, and that $\delta$ can be chosen arbitrarily small, we conclude as claimed that $p (x,y)$ converges to 1 as $x \to +\infty$ uniformly with respect to $y \in \R^{N-1}$. 

Furthermore, the above inequality also implies that, for any $L >0$, then
\begin{equation}\label{pinf}
\inf_{x \in (-L,L), y \in \R^{N-1}} p(x,y) > 0.
\end{equation}
Indeed, if there exists sequences $x_n \in (-L,L)$ and $y_n \in \R^{N-1}$ such that $p(x_n,y_n) \to 0$ as $n\to+\infty$, then by a strong maximum principle argument, the sequence $p(x_n+ x,y_n +y)$ also converges locally uniformly to~0, which contradicts the inequality~\eqref{plim1b}. 

Finally let us prove the following claim, using a sliding argument.
\begin{claim}\label{claim13}We have that
 $p\equiv p^+_c$ the maximal steady state. 
 \end{claim}
 First we will find $\overline{X}>0$ such that, for all $(x,y) \in \Omega$,
\begin{equation}\label{p+cslide}
p^+_c (x-\overline{X},y)\leq p(x,y).
\end{equation}
To do so, let $\delta >0$ and $L>0$ be respectively small and large enough so that:
$$ \forall x \geq L, \ \forall y \in \omega,  \quad p (x,y) \geq 1 - \delta \quad \mbox{ and } \quad \forall x \leq -L, \ , \quad p^+_c (x,y) \leq \delta,$$
$$\forall x \geq L, \ \forall y \in \omega, \  \forall s \geq 1 -\delta, \ \ \partial_s f(x,y,s) < \frac{g_+ '(1)}{2}  <0 ,$$
$$ \forall x \leq -L, \  \forall y \in \omega, \ \forall s \leq \delta, \  \ \partial_s f(x,y,s) < \frac{ g_- ' (0)}{2} < 0.$$
Since $p^+_c (x,y)$ converges to 0 as $x \to -\infty$ uniformly with respect to~$y$, and using also \eqref{pinf}, clearly we can find $\overline{X}$ large enough so that the inequality~\eqref{p+cslide} holds strictly for any $(x,y) \in [-L,L] \times \omega$. 

Let us now prove that the inequality also holds when $x \geq L$, using a weak maximum principle. Indeed, on this part of the domain the function $q(x,y) = p^+_c (x-\overline{X},y) - p(x,y)$ satisfies
\begin{eqnarray*}
0& =& \Delta q + c \partial_x q +  \frac{f(x- \overline{X},y,p^+_c (x-\overline{X},y)) - f (x,y,p(x,y))}{q (x,y)} q \\
& \leq & \Delta q + c \partial_x q +  \frac{f(x,y,p^+_c (x-\overline{X},y)) - f (x,p(x,y))}{q (x,y)} q.
\end{eqnarray*}
Note that $  \frac{f(x,y,p^+_c (x-\overline{X},y)) - f (x,y,p(x,y))}{q (x,y)}$ can be extended by $\partial_s f(x,p)$ when $q= 0$ so that it is in $L^\infty (\Omega)$. Moreover, whenever $q >0$, then from our choice of $L$:
$$\frac{f(x,y,p^+_c (x-\overline{X},y)) - f (x,y,p(x,y))}{q (x,y)}  \leq \frac{g_+ ' (1)}{2} < 0.$$
Because $q (x=L) \leq 0$ and $\lim_{x \to +\infty} q (x,y) = 0$ uniformly with respect to $y$ (see Proposition~\ref{p+c_exists} and Claim~\ref{claim12}), it easily follows from 
the weak maximum principle that $q$ may not reach a positive value, hence~\eqref{p+cslide} holds for all $x \geq L$ and $y \in \omega$. The same argument as above applies when $x \leq -L$, so that we omit the details and~\eqref{p+cslide} is proved.

Now define 
$$X^*=\inf \{ X>0 \ | \  p^+_c(x - X,y)-p(x,y)\leq 0 \ \ \forall (x,y)\in\Omega \}.$$
By the continuity of $p^+_c$ and the definition of $X^*$, we know that $p^+_c (\cdot -  X^*,\cdot) - p(\cdot) \leq 0$ and that there exists $(x_n,y_n)_n$ such that 
\be\lim_{n\to+\infty} p^+_c (x_n- X^*,y_n)-p(x_n,y_n)=0.\ee
We either have $|x_n|<+\infty$ and it converges to some constant $x_\infty \in\R$ (up to a subsequence) or $|x_n|\to+\infty$. In the first case considering  $q_n(x,y)=p^+_c (x - X^*,y+y_n)-p(x, y + y_n)$, we know from parabolic estimates that $q_n\to q_\infty$ as $n\to+\infty$ where similarly as above $q_\infty $ is a subsolution of a linear elliptic equation with bounded coefficients. Moreover, by construction we have that $q_\infty \leq 0$ and $q (x_\infty,0)=0$. By the strong maximum principle 
we get that $q_\infty\equiv 0$ and thus
$$\lim_{n\to +\infty} p^+_c( \cdot - X^*,\cdot +y_n)= \lim_{n \to +\infty} p (\cdot, \cdot + y_n).$$
The right-hand side (resp. left-hand side)  is a stationary solution (resp. a subsolution) of Problem~\eqref{problemmoving} where $f$ is replaced by 
$$\hat{f} (\cdot) = \lim_{n \to +\infty} f (\cdot , \cdot +y_n,\cdot) ,$$
this limit being understood in the $L^\infty_{loc}$-weak star topology, up to extraction of another subsequence. This is only possible if $X^*=0$, otherwise the left-hand side does not satisfy the equation: indeed one can check that $\hat{f}$ still satisfies Hypotheses~A, and in particular it is monotonic with respect to the space variable~$x$ and converges to $g_\pm$ as $x \to \pm \infty$. Note that in the case when $y_n$ is bounded, in particular when $\omega$ is bounded, it is not needed to shift in the $y$ direction which makes this part of the argument simpler. Nevertheless, in the case when $|x_n|< +\infty$ we have proved that $p^+_c \leq p$, hence $p^+_c\equiv p$.

Next consider the case when $|x_n|\to+\infty$ as $n\to+\infty$. Assume by contradiction that $X^* >0$, and note that thanks to the above we know that 
$$\inf_{x \in [-L,L], y \in \omega} p_c^+ (x-X^*,y) - p(x,y) < 0.$$
By standard estimates, the first derivative of $p_c^+$ with respect to $x$ is uniformly bounded, so that there exists $\varepsilon>0$ small enough such that $p^+_c(x-(X^*-\varepsilon),y)-p(x,y)<0$ for all $x \in[-L,L]$  and $y \in \omega$. Proceeding as before, we can show by a weak maximum principle that $p^+_c(x-(X^* - \varepsilon),y) \leq p(x,y)$ for all $x \in \mathbb{R}$ and $y \in \omega$, which contradicts the definition of $X^*$. Again we conclude that $X^* \equiv0$, therefore $p^+_c\equiv p$. 
Claim~\ref{claim13} as well as Proposition~\ref{otherstates} is proved.
\end{proof}

Now using Lemma~\ref{lemma:maxprinciple} from the Appendix, we can reduce the number of steady states satisfying some specific exponential bound at $+\infty$. It will turn out in Section~\ref{sec:sharpthresholds} that the solution of equation~\eqref{problemmoving} either spreads or satisfies a similar exponential bound. This is why it is useful to study such steady states. Proposition~\ref{othercritical} below shows that these fast decaying (or ground) states cannot be ordered. If it weren't true, then some localized perturbation of a ground state would still be grounded (in the sense of Definition~\ref{def:dyn}). Thus, Proposition~\ref{othercritical} can formally be understood as a necessary condition for the sharp threshold phenomena which our main results exhibit.
\begin{Prop}\label{othercritical}
Assume that $c \geq 2 \sqrt{g_+'(0)}$. Then there does not exist $p_1 \leq \not \equiv p_2$ with $$\inf_{y \in \omega} p_2 (0,y) - p_1 (0,y) >0$$ and $p_1$, $p_2$ are two positive ground states, i.e. steady states of \eqref{problemmoving} satisfying the exponential estimates
$$o(e^{-\frac{c}{2} x})\quad  \left(\mbox{if }c > 2 \sqrt{g_+ '(0)} \, \right),$$
$$o((1+\sqrt{x}) e^{-\frac{c}{2} x})\quad \left(\mbox{if }c = 2 \sqrt{g_+ '(0)} \, \right),$$
as x $\to +\infty$ uniformly with respect to $y \in \omega$.

If $c \geq c^*$, then there exist no positive ground steady state $p$, i.e. satisfying the above exponential bound.
\end{Prop}
We immediately note that, thanks to the strong maximum principle, the condition $\inf_{y \in \omega} p_2 (0,y) - p_1 (0,y) >0$ can be removed when $\omega$ is bounded. The second part of Proposition~\ref{othercritical} will also be useful in order to study the large time behaviour of solutions of \eqref{problemmoving} when $c = c^*$, regardless of the choice of the cross section~$\omega$.
\begin{proof}
We prove the first statement and consider $p_1 \leq p_2$ to be two such steady states. It is straightforward to prove that, for any $X >0$, then
$$\inf_{x \in [-X,X], y \in \omega} p_2 (x,y) - p_1 (x,y) >0.$$
In particular, one can find some small $\eta >0$ such that
$$\forall (x,y) \in [-X,X] \times \omega, \quad p_1 (x -\eta, y) \leq p_2 (x,y).$$
Choosing $X$ large enough and since any shift to the right of $p_1$ is a subsolution of \eqref{problemmoving}, it follows from Lemma~\ref{lemma:maxprinciple} that $p_1 (x- \eta,y) \leq p_2 (x,y)$ for any $x \geq X$ and $y \in \omega$. In a similar fashion, since any steady state converges to 0 as $x\to -\infty$ uniformly with respect to $y$ (recall that $p^+_c$ lies above any other stationary state) and since $f$ has a negative derivative with respect to $s$ far to the left, one can also show that $p_1 (x- \eta,y) \leq p_2 (x,y)$ for any $x \leq -X$ and $y \in \omega$. Indeed, if we assume that there exists $(x_0,y_0)\in\Omega$ such that $x_0<-X$ and $p_1(x_0-\eta,y_0)-p_2(x_0,y_0)>0$, we can assume without loss of generality that it achieves a positive maximum (up to convergence via a subsequence if needed) and we reach a contradiction.

Now define
$$\eta^* := \max \{ \eta > 0 \ | \ p_1 (\cdot - \eta, \cdot) \leq p_2 (\cdot) \}.$$
It is clear that $\eta^*$ is finite as $p_1$ is positive and $p_2$ decays to 0 as $x \to +\infty$. Then there exists some sequence $(x_n,y_n)$ such that $p_1 (x_n  - \eta^*,y_n) - p_2 (x_n,y_n) \to 0$ as $n \to +\infty$. If this sequence is bounded, then we get some $(x_\infty,y_\infty) \in \overline{\Omega}$ such that $p_1 (x_\infty - \eta^* , y_\infty) = p_2 (x_\infty, y_\infty)$, which by the maximum principle implies $0< p_1 (\cdot - \eta^*,\cdot) \equiv p_2 (\cdot)$. Together with the monotonicity of $f$ in the space variable and the fact that $\eta^* >0$, this is a contradiction.

Next, assume that the sequence $x_n$ is bounded but not $y_n$ (in the case $\omega=\R^{N-1}$). Thanks to the boundedness of~$f$ and parabolic estimates, one can define the functions $p_{1,\infty} = \lim_n p_1 (\cdot  - \eta^*,\cdot + y_n)$ and $p_{2,\infty} = \lim_n p_2 (\cdot , \cdot + y_n)$ and check that
$$\Delta p_{1,\infty} + c\partial_x p_{1,\infty} + \hat{f} (x-\eta^*,y,p_{1,\infty})=0,$$
$$\Delta p_{1,\infty} + c\partial_x p_{2,\infty} + \hat{f} (x-,y,p_{2,\infty})=0,$$
where $\hat{f}$ is the limit as $n \to +\infty$ of $f(\cdot,\cdot + y_n,\cdot)$ in the $L^\infty_{loc}$-weak star topology. Now by the strong maximum principle we get that $p_{1,\infty} \equiv p_{2,\infty}$. Note that here, the monotonicity of $f$ (hence of $\hat{f}$) and the positivity of~$\eta^*$ does not immediately give a contradiction, as a priori $p_{1,\infty} \equiv p_{2,\infty} \equiv 0$ may occur. However, this is ruled out from our assumption that $\inf p_2 (0,\cdot) >0$, thus $p_{2,\infty}$ cannot be trivial.

It remains to consider the case when $|x _n | \to +\infty$ and $\inf_{x \in [-X,X], y \in \omega} p_2 (x,y) - p_1 (x - \eta^*,y) >0$ for any $X>0$. Choose $\varepsilon >0$ small enough so that 
$$\inf_{y \in \omega} p_2 (X,y) - p_1 (X-\eta^* - \varepsilon, y) >0.$$
Using Lemma~\ref{lemma:maxprinciple}, we get that 
$$p_1 (\cdot - \eta^* - \varepsilon,\cdot) < p_2 (\cdot) \quad \mbox{ in $\{ x \geq X\} \times \omega$}.$$
On the other hand, $p_1$ and $p_2$ both converge to 0 as $x \to -\infty$ (they are bounded from above by $p_c^+$), and $\lim_{x \to -\infty} \partial_s f(x,s) = g_- ' (s)$ which is negative in a neighborhood of $0$. Therefore, up to increasing $X$, the usual weak maximum principle can be used to infer also that 
$$p_1 (\cdot - \eta^* - \varepsilon,\cdot) < p_2 (\cdot) \quad \mbox{ in $\{ x \leq - X\} \times \omega $}.$$
Finally we contradicted the definition of $\eta^*$. Such a pair of steady states does not exist.\medskip

Let us now prove the second statement. Assume that $c\geq c^*$ and, proceeding by contradiction, let $p$ be a positive steady state satisfying the appropriate exponential bound. From standard phase plane analysis, it is well-known (see Aronson and Weinberger \cite{AW}, Prop 4.1 and Prop 4.4) that the equation
\begin{equation}\label{homhom}
v'' + c v' + g_+ (v) = 0
\end{equation}
admits a unique (up to shift) decreasing stationary solution $v$ decaying to 0 as $x \to +\infty$ and whose associated trajectory in the phase plane is extremal, in the sense that any other trajectory going through $(0,0)$ must lie above it. Moreover, it is also known by standard ODE technics~\cite{cl} that $v$ has the asymptotics
$$v (x) \sim A e^{-\displaystyle \frac{c + \sqrt{c^2 - 4 g_+ ' (0)}}{2}\, x},$$
where $A>0$. By construction, the trajectory associated with~$v$ must be below the trajectory associated with the travelling wave with speed $c$. Thus the solution $v$ may either coincide with this travelling wave (this can only occur if $c = c^*$), in which case $v (-\infty)= 1$, or satisfy $v (-\infty) = +\infty$. 

Let us now note that the real number $X$ in Lemma~\ref{lemma:maxprinciple} only depends on the function $f$, and on the decay of the sub and supersolutions $\underline{\phi}$, $\overline{\phi}$ to $0$ as $x \to +\infty$. Indeed, this can easily be seen from the proof of Lemma~\ref{lemma:maxprinciple}. In particular, we can choose $X$ so that Lemma~\ref{lemma:maxprinciple} can be applied for any supersolution $\overline{\phi} \leq p$ and subsolution $\underline{\phi} \leq p$.

From the exponential bound on $p$ and the fact that $p \leq p^+_c$, it is straightforward that $\|p\|_\infty < 1$. Then, we can shift the function $v$ without loss of generality so that
$$v(x) \geq p (x,y)$$
for all $x \leq X$ and $y \in \omega$, as well as, using the fact that $p >0$,
\begin{equation}\label{infimum1}
\inf_{x \leq X, y \in \omega}  (v-p) =0.
\end{equation}
We now define
$$\overline{\phi} := \min \{ v,p\} \leq p,$$
which from the above satisfies that 
$$\overline{\phi} (x,y) = p(x,y)$$
for all $x \leq X$ and $y \in \omega$. Clearly $\overline{\phi}$ is a (generalised) supersolution of~\eqref{problemmoving}. Letting also $\underline{\phi} \equiv p$ and applying Lemma~\ref{lemma:maxprinciple}, we conclude that
$$\overline{\phi} \equiv p.$$
In other words, $p \leq v$ in the whole domain. 

However, $p$ is a subsolution of~\eqref{homhom}. Therefore, if the infimum in~\eqref{infimum1} is reached at some point $(x_0,y_0) \in \Omega$, then the strong maximum principle implies that $p\equiv v$, which is clearly impossible. Up to some appropriate shift, the same argument applies even if the infimum in~\eqref{infimum1} is not reached. Indeed, let a sequence $(x_n,y_n)_n$ be such that $v(x_n,y_n) - p (x_n,y_n) \to 0$ as $n \to +\infty$. Due to the asymptotics of $v$ and $p$ as $x \to -\infty$, we have that $x_n$ must remain bounded. By standard parabolic estimates, the sequence $p(\cdot, \cdot +y_n)$ converges and, by the strong maximum principle, the limit must be identically equal to~$v$. This is again an obvious contradiction. We conclude that such a positive steady state~$p$ does not exist. The proposition is proved.
\end{proof}

\section{Large time behaviour: three speed ranges}\label{sec:3speed}
	
In this section, we prove Theorem \ref{th:regime} and highlight the existence of three different regimes for the speed $c$ which will determine the large time behaviour of the solutions of~\eqref{problemmoving}-\eqref{initcond}.

For convenience, we start with a preliminary lemma which shows that as time goes to infinity any bounded solution of \eqref{problemmoving} must lie below the maximal steady steate $p^+_c$, which we defined rigorously in Section~\ref{sec:classStat}.
\begin{Lemma}\label{pratique}
Let $u_0$ be a bounded nonnegative initial datum, and $u$ be the corresponding solution of \eqref{problemmoving}. Then 
$$\lim_{t \to +\infty} \sup_{x \in \mathbb{R}, y \in \omega} \left( u(t,x,y) - p^+_c (x,y)\right) =0.$$
\end{Lemma}
\begin{proof}According to the proof of Proposition~\ref{p+c_exists}, for any $M> \max \{1, \|u_0\|_\infty\}$, the solution of \eqref{problemmoving} with initial datum $M$ is decreasing with respect to time and converges locally uniformly to $p^+_c$ as $t \to +\infty$. It immediately follows by comparison that, for any compact set $K \subset \overline{\Omega}$,
\begin{equation}\label{pratique2}
\limsup_{t\to+\infty} \sup_{(x,y) \in K} \left( u(t,x,y) - p^+_c (x,y) \right) \leq 0.
\end{equation}
We need to improve this estimate and prove that this inequality still holds when replacing $K$ by $\Omega$. 

Proceed by contradiction and assume that there exist some $\varepsilon>0$ and some sequences $t_n \to +\infty$ and $(x_n,y_n) \in \Omega$ such that
\begin{equation}\label{pratique1}
u(t_n,x_n,y_n) \geq p^+_c (x_n,y_n) + \varepsilon.
\end{equation}
Assume first that $(x_n)_{n }$ is bounded. Up to extraction of a subsequence, $u(t_n+ t,x_n+x,y_n+y)$ and $p^+_c (x_n+x,y_n+y)$ converge to $u_\infty$ and $p_\infty$ respectively an entire and a stationary solution of \eqref{problemmoving}, where $f$ and $\omega$ are replaced by some $\hat{f}$ and $\hat{\omega}$ which satisfy the same assumptions. Moreover, it follows from Proposition~\ref{p+c_exists} that $p_\infty (+\infty, \cdot) =1$, hence it is a maximal steady state. By \eqref{pratique1} we also get that $u(0,0,0) \geq p_\infty (0,0) +\varepsilon$ and, repeating the argument above, we reach a contradiction.

Next, assume that $x_n \to + \infty$. We again extract a converging subsequence from $u(t_n+t,x_n+x,y_n+y)$ and obtain in the limit an entire solution $u_\infty$ of the monostable equation
$$\partial_t u_\infty = \Delta u_\infty + c u_\infty + g_+ (u_\infty)$$
satisfying $u_\infty (0,0,0) > 1$ because of \eqref{pratique1}. This is also a contradiction.

Similarly, if $x_n \to -\infty$ one gets a positive and bounded entire solution $u_\infty$ of the equation
$$\partial_t u_\infty = \Delta u_\infty + cu_\infty + g_- (u_\infty).$$
We have reached a contradiction in all cases, thus \eqref{pratique2} holds when the set $K$ is replaced by the whole domain~$\Omega$. Finally, since $p^+_c \to 0$ as $x \to -\infty$, it is clear that $\sup_{(x,y)\in \Omega} u(t,x,y) - p^+_c (x,y) \geq 0$ for all $t >0$, and therefore Lemma~\ref{pratique} is proved.\end{proof} 

\subsection{Spreading when $c<2\sqrt{g_+ '(0)}$}\label{sec:low}

In this first speed range, we construct an arbitrarily small compactly supported subsolution of Problem~\eqref{problemmoving}, from which spreading will follow for any nontrivial initial datum. Knowing that $c<2\sqrt{g_+ '(0)}$, there exist $\delta>0$ small and $X_\delta>0$ large enough such that 
$$f (x,y,s) \geq g_\delta (s)  \quad \forall \: x \geq X_\delta,\: y\in\omega, \: s\in\R ,$$
$$\partial_s f (x,y,s) \geq g'_\delta (0)  \quad \forall \: x \geq X_\delta,\:y\in\omega, \: s\in [0,\delta] ,$$
where $g_\delta$ is a monostable function whose zeros are 0 and $1-\delta$, and such that $c<2\sqrt{g_\delta'(0)}$. In the case when $\omega = \R^{N-1}$, we consider the principal eigenvalue problem
\be\label{pb:Dirichleteigen}\begin{cases}
\Delta \phi_{R,c} + c\partial_{x}\phi_{R,c} + g_\delta'(0)\phi_{R,c} =\Lambda_{R,c}\phi_{R,c} &\text{in } B_R,\\
\phi_{R,c} >0 &\text{in }B_R,\\
\phi_{R,c} = 0 &\text{on } \partial B_R.
\end{cases}\ee
where $R>0$ and $B_R$ denotes the ball of radius $R$ centered at 0. 
\begin{Rk}
For the case $\omega$ bounded we instead consider the principal eigenvalue problem
\be\label{pb:Dirichleteigen_bdd}\begin{cases}
 \phi_{R,c} ''  + c\partial_{x}\phi_{R,c} + g_\delta'(0)\phi_{R,c} =\Lambda_{R,c}\phi_{R,c} &\text{in }(-R,R),\\
\phi_{R,c} >0 &\text{in } (-R,R),\\
\phi_{R,c} = 0 &\text{on } \{-R,R\}.
\end{cases}\ee
Then the proof below applies to the letter.\end{Rk}
Computing that $\psi = e^{\frac{c}{2} x} \phi_{R,c}$ satisfies
$$\Delta \psi = \left( \Lambda_{R,c} + \frac{c^2}{4} - g'_\delta (0) \right) \psi$$
one can deduce that $\Lambda_{R,c} >0$ for $R$ large enough. Then, extending $\phi_{R,c}$ by 0 outside of the ball, one can check that for any $\kappa>0$ small enough, the function $\kappa \phi_{R,c} (x - X_\delta -R,y)$ is a subsolution of \eqref{problemmoving}. In particular, the solution $\underline{u}$ of~\eqref{problemmoving} with initial datum $\underline{u} (t=0) = \kappa \phi_{R,c}$ is increasing in time. As it is also bounded by some constant $M>1$, it converges locally uniformly to some positive and bounded stationary state~$p$ as $t\to+\infty$. 

Let us prove that $p \equiv p^+_c$. By Proposition~\ref{otherstates} it is enough to show that $p$ does not converge to 0 as $x \to +\infty$ locally uniformly with respect to~$y$. We proceed by contradiction. If $p$ converges to 0, and since by construction $p \geq \kappa \phi_{R,c}$, clearly there exists some spatial shift $X>0$ such that
$$p(x,y) \geq \kappa \phi_{R,c} (x- X,y)$$
for all $(x,y) \in \R^{N}$, and the equality holds at some point $(x_0,y_0)$. Here we also used the fact that $\phi_{R,c}$ is compactly supported. As any shift to the right of $\kappa \phi_{R,c}$ is again a subsolution of~\eqref{problemmoving}, a strong maximum principle implies that $p (\cdot) \equiv \kappa \phi_{R,c} (\cdot -X,\cdot)$, which of course is impossible. We conclude that $p \equiv p^+_c$.

\begin{Rk}
We note here that the argument in the last paragraph also proves that 0 and $p^+_c$ are the only two nonnegative and bounded steady states of \eqref{problemmoving}, as was announced in the previous section.
\end{Rk}

Finally, for any nonnegative and nontrivial initial datum~$u_0$, by the parabolic strong maximum principle we have that $u(1,\cdot) >0$. In particular, one can choose $\kappa$ small enough so that $u(1,\cdot) \geq \kappa \phi_{R,c}$, and it easily follows by comparison that 
for any compact set $K \subset \overline{\Omega}$, one has
$$\liminf_{t \to +\infty} \inf_{(x,y) \in K} u(t,x,y) - p^+_c (x,y) \geq 0.$$
Putting this together with Lemma~\ref{pratique}, we conclude that $u(t,\cdot)$ converges locally uniformly to $p^+_c$ as $t \to +\infty$, i.e. spreading occurs.

\subsection{Vanishing and spreading when $c\in[2\sqrt{g_+ '(0)},c^*)$}\label{sec:interm}

We begin by proving that spreading can occur in this intermediate speed range. As in the previous subsection, we let some small $\delta >0$ and large $X_\delta >0$ be such that, for all $x \geq X_\delta$, $y\in\omega$ and $s \geq 0$, one has $f(x,y,s) \geq g_\delta (s)$ where $g_\delta$ is a monostable type function whose zeros are 0 and $1-\delta$, and whose minimal wave speed $c^*_\delta$ is close enough to $c^*$ so that $c^*_\delta >c$.

Then it is well known that there exists a compactly supported subsolution $v^\delta$ of
$$ \Delta v^\delta + c \partial_{x} v^\delta + g_\delta (v^\delta) \geq 0.$$
Such a subsolution can be constructed by phase plane analysis, as we do below. The subsolution we use here is adapted to our situation where the favourable zone is receding as time increases; however it is quite similar to the one in~\cite{AW} and therefore we omit some details. Let $\rho$ and $c_1$ be two positive constants such that $c+(n-1)/\rho<c_1<c^*_\delta$ and $q$ be the solution of 
$$q '' + c_1 q' + g_\delta (q)=0, \: q(0)=\eta<1-\delta,\:q'(0)=0.$$
Since $c_1<c^*_\delta$, we know from~\cite{AW} that if $\eta$ is close enough to $1-\delta$, then there exists $b>0$ such that $q(b)=0$ and $q>0$, $q'<0$ in $(0,b)$.
We then define
\begin{equation*}
v^\delta(x,y)=\begin{cases}
\eta &\forall \: ||(x,y)||<\rho\\
q(||(x,y)||-\rho) &\forall\: ||(x,y)||\in[\rho,\rho+b],\\
0 &\forall \: ||(x,y)||>\rho+b.
\end{cases}\end{equation*}
for all $(x,y)=\R^{N}$. One can check that, for any $X>0$ large enough, the function $v^\delta (x - X, y)$
is a subsolution of \eqref{problemmoving}. 
\begin{Rk}\label{rk:yshift}
For the case $\omega$ bounded then replace $q (\|(x,y)\|-\rho)$ by $q(|x|-\rho)$ in the definition of $v^\delta$ above, so that $v^\delta$ no longer depends on~$y$ and clearly satisfies the Neumann boundary condition.

Note also that in the case $\omega = \R^{N-1}$, any shift of $v^\delta$ in the $y$-direction is also a subsolution of~\eqref{problemmoving}. This will prove useful later on.
\end{Rk} 
It follows from the parabolic maximum principle that the solution with initial datum $v^\delta (\cdot-X,\cdot)$ is increasing in time. It is also uniformly bounded by the constant $\max \{ \|v^\delta\|_\infty, 1 \}$ and therefore, by standard parabolic estimates it converges locally uniformly to some steady state $p$. Let us prove that $p \equiv p^+_c$. If not, then $p$ converges as $x \to +\infty$ to 0 locally uniformly with respect to $y$, due to Proposition~\ref{otherstates}. Since $v^\delta (\cdot - X,\cdot)< p (\cdot)$ by construction, it is straightforward that there exists 
$$X^* = \sup \{ X' \geq X \ |\ v^\delta (\cdot - X' ,\cdot) < p (\cdot) \} < +\infty.$$
By continuity and the fact that $v^\delta$ is compactly supported, we get that $v^\delta (\cdot - X^*,\cdot) \leq p(\cdot)$, and the equality holds for some $(x_0,y_0) \in \Omega$. Using the strong maximum principle, we reach a contradiction, hence $p \equiv p^+_c$. We have proved that spreading occurs for the particular initial datum $v^\delta$. By comparison and thanks to Lemma~\ref{pratique}, it easily follows that spreading also occurs for any $u_0 \geq v^\delta$.\\

We next turn to the vanishing scenario and construct a supersolution.
Assuming first that $c>2\sqrt{g_+'(0)}$, we define the function $v(x)=e^{-(\frac{c}{2}+\delta)x}$. Then 
\begin{align*}
\displaystyle \partial_t v-\partial_{xx}v-c\partial_xv-f(x,y,v)&=\displaystyle \left(\frac{c^2}{4}-\delta^2 \right) v -f(x,y,v)\\
&\geq \left(\frac{c^2}{4}-\delta^2 \right) v -g_+(v)\\
&= \left(\frac{c^2}{4}-\delta^2+g'_+(s)\right)v,
\end{align*}
for some $s\in(0,v)$. Up to increasing $X$ we know that $v(x)\leq\varepsilon$ for all $x\geq X$, thus choosing $\delta>0$ small enough such that $\frac{c^2}{4}-\delta^2>g'_+(s)$ for all $s\in(0,\varepsilon)$, we conclude that $v$ is a supersolution of \eqref{problemmoving} in $[X,+\infty) \times \omega$.

When $c=2\sqrt{g'_+(0)}$, we define $v(x)=(1+x^\beta)e^{-\frac{c}{2}x}$, where $0<\beta<1$. Then we can choose $X$ large enough so that
$$\partial_x v (X) < 0$$
and 
\begin{align*}
\partial_t v-\partial_{xx} v -c\partial_xv-f(x,y,v)&=-\beta(\beta-1)x^{\beta-2}e^{-\frac{c}{2}x}+ \frac{c^2}{4}v-f(x,y,v)\\
&\geq -\beta(\beta-1)x^{\beta-2}e^{-\frac{c}{2}x}+(g'_+(0)-g_+'(s))v\\
&\geq -\beta(\beta-1)x^{\beta-2}e^{-\frac{c}{2}x}-C_rv^{1+r}\\
&\geq 0,
\end{align*}
in $[X,+\infty) \times \omega$. As above, in the first inequality $s \in (0,v)$ depends on the value of $v$ at the corresponding point. In the second inequality, the constant $C_r$ comes from the $C^{1,r}$ regularity of $g_+$ around~0, and from the fact that $v (x) \to 0$ as $x \to +\infty$.

Now in both cases, choosing $A = v(X)$, the function
$$\overline{u}_0 (x)= \left\{ 
\begin{array}{ll}
A & \mbox{if } x \leq - X,\vspace{3pt}\\
v(x+2X) & \mbox{if } x \geq -X,
\end{array}
\right.
$$
is a generalised supersolution of \eqref{problemmoving} with Neumann boundary condition when $\Omega$ has a boundary, up to increasing $X$ again and thanks to the fact that $f$ is negative far to the left. Therefore the solution $\overline{u}$ of \eqref{problemmoving} with initial datum $\overline{u}_0$ is nonincreasing in time and, by standard parabolic estimates, it converges locally uniformly to some nonnegative stationary state~$p$ as $t\to+\infty$. 

Let us check that $p \equiv 0$. We follow the same argument as in the proof of Proposition~\ref{othercritical} (choosing $\beta<1/2$ when $c=2\sqrt{g_+'(0)}$). The main idea here is to observe that the first conclusion of Proposition~\ref{othercritical} still holds when $p_2$ is no longer a steady state but a stationary supersolution of \eqref{problemmoving}. Note first that $\overline{u}_0$ satisfies the appropriate exponential bound and does not depend on $y \in \omega$, and therefore $p$ is a ground state of \eqref{problemmoving} in the sense we introduced in Proposition~\ref{othercritical}. Proceed by contradiction and assume that $p$ is not identically equal to~0. Since $p \leq \overline{u}_0$, then there exists
$$\eta^* = \max \{ \eta \geq 0 \ | \ p (\cdot- \eta,\cdot) \leq \overline{u}_0 (\cdot) \}.$$
Because $\overline{u}_0$ is a supersolution of \eqref{problemmoving} and satisfies, as we mentioned above, the same exponential bound as in Proposition~\ref{othercritical}, the strong maximum principle and Lemma~\ref{lemma:maxprinciple} lead to a contradiction. 

It only remains to show that the convergence of $\overline{u}$ to $0$ as $t \to +\infty$ is also uniform with respect to the space variables. Assume that this is not true, so that there exist $\varepsilon >0$ and some sequences $t_n \to +\infty$ and $(x_n,y_n) \in \Omega$ such that
$$\overline{u} (t_n,x_n,y_n) \geq \varepsilon$$
for all $n \in \mathbb{N}$. Moreover, the sequence $(x_n)_n$ must be bounded. This follows on one hand from the fact that $\overline{u}_0 (x)$ converges to 0 as $x \to +\infty$, and on the other hand from Lemma~\ref{pratique} where $p^+_c (x,y) \to 0$ as $x \to -\infty$. When $\omega$ is bounded, we have already reached a contradiction since we have shown that $\overline{u}$ converges locally uniformly in space to $0$ as $t \to +\infty$. When $\omega =\R^{N-1}$, then up to extraction of a subsequence, we have that $x_n \to x_\infty \in \mathbb{R}$ and $\overline{u} (t+t_n,x,y+y_n)$ converges to an entire solution $u_\infty$ of some equation
\begin{equation}\label{eq:shift0}
\partial_t u_\infty = \Delta u + c\partial_x u_\infty + \hat{f}(x,y,u_\infty)
\end{equation}
where $\hat{f}$ satisfies Hypotheses~A. Note that to obtain~\eqref{eq:shift0}, one must appropriately shift the equation~\eqref{problemmoving} and pass to the limit first in the $L^\infty_{loc}$-weak star topology. In any case, we then repeat the same argument as above. First, $u_\infty$ still lies below $\overline{u}_0$. Because $\overline{u}_0$ is also a supersolution of this shifted equation, one gets that $u_\infty \leq \hat{p}$ where $\hat{p}$ is a (possibly trivial) ground state of equation~\eqref{eq:shift0}. Using again Lemma~\ref{lemma:maxprinciple}, one can infer that $\hat{p} \equiv 0$, which contradicts the fact that $u_\infty (0,x_\infty,0) \geq \varepsilon$. We conclude that vanishing occurs for the initial datum $\overline{u}_0$. This immediately extends to any initial datum below $\overline{u}_0$.

\subsection{Vanishing when $c \geq c^*$}\label{sec:high}

We now deal with the higher speed range $[c^*,+\infty)$. First assume that $c > c^*$, where the proof naturally follows from the fact that the solution $v(t,x,y)$ of the homogeneous monostable problem
$$\partial_tv-\Delta v=g_+(v)$$
with initial datum $u_0$ spreads with speed $c^*$. More precisely, it is known that for any $X \in \mathbb{R}$,
$$\limsup_{t \to +\infty} \sup_{x \geq X, y \in \omega} v(t,x+ct,y) = 0,$$ 
This was proved in the celebrated work of Aronson and Weinberger~\cite{AW} when the initial datum is less than 1. Nevertheless, it is standard to extend this result to general initial data when $g_+ (s) < 0$ for all $s >1$, using for instance a supersolution of the type of \eqref{eq:supsol11} below. 

Since $f (x,y,s) \leq g_+ (s)$ for all $x \in \mathbb{R}$, $y\in\omega$ and $s \geq 0$, it immediately follows by comparison that the solution $u$ converges to 0 as $t \to +\infty$, uniformly in any set $\{x \geq -X \} \times \omega$. Putting this together with Lemma~\ref{pratique} and the fact that $p^+_c (x,y) \to 0$ as $x\to -\infty$ uniformly with respect to $y \in \omega$, we conclude that $u$ converges uniformly to 0 in large time, i.e. it vanishes in the sense of Definition~\ref{def:dyn}.\\

The case $c = c^*$ is slightly more complicated since the solution $v$ of the monostable homogeneous equation may actually persist in the moving frame with speed $c^*$ when $c^* > 2 \sqrt{g_+'(0)}$~(see~\cite{rothe}). However vanishing again occurs for solutions of \eqref{problemmoving}, which follows from Proposition~\ref{othercritical} as we will see below.

As we mentioned in Section \ref{sec:classStat}, for any $c \geq 2 \sqrt{g_+ '(0)}$, the monostable equation 
\begin{equation}\label{eq:monob}
v'' + cv' + g_+ (v) = 0
\end{equation}
admits a unique (up to shift) solution behaving as
$$v(x) \sim A e^{\displaystyle -\frac{c + \sqrt{c^2 - 4 g_+ ' (0)}}{2} \, x}$$
as $x \to +\infty$, where $A>0$. Moreover, since $c = c^*$, there exists a travelling wave connecting 1 to 0, and the trajectory in the phase plane of~\eqref{eq:monob} associated with this travelling wave must be above the trajectory associated with~$v$. In other words, either $v$ coincides with the travelling wave with minimal speed, or $v$ converges to $+\infty$ as $x \to -\infty$.

In the latter case, we consider $\overline{u}$ the solution of \eqref{problemmoving} with initial datum 
$$ \overline{u}_0  = \min \{ \|u_0\|_\infty, v\}.$$
Since $v$ is positive and satisfies $v(-\infty) = +\infty$, then for any initial datum $u_0$ as in~\eqref{initcond} we can shift $v$ (without loss of generality) so that $u_0 \leq \overline{u}_0$. As $\overline{u}_0$ is a supersolution, we know that $\overline{u}$ is decreasing in time and therefore it converges (at least) locally uniformly to a nonnegative stationary state $p \leq v$. In particular, $p$ is a ground state and, by Proposition~\ref{othercritical}, we have that $p \equiv 0$. By the comparison principle, the solution $u$ of \eqref{problemmoving} with initial datum $u_0$ also converges locally uniformly to~0. Proceeding as in the vanishing scenario of the previous subsection, this convergence is in fact uniform. In other words, vanishing always occurs.

Assume now that $v$ is the travelling wave with minimal speed (this may only happen when $c= c^* = 2\sqrt{g_+ '(0)}$). Up to some shift we assume that $v(0)=1-\frac{\eta}{2}$ where $\eta>0$ is small enough so that $\eta\leq -g_+ '(1)/4$ and $g_+'(s)\leq g_+'(1)/2<0$ for all $s\in(1-\eta,1+\eta)$. As we do not assume that $u_0 < 1$, we cannot directly compare the solution with $v$. Instead, we introduce
\begin{equation}\label{eq:supsol11}
\overline{v} (t,x) = v (x - c^* t -  \gamma (1 - e^{-\eta t} )) + \eta e^{-\eta t} \varphi (x - c^* t - \gamma (1-e^{-\eta t})),
\end{equation}
with $\varphi$ a smooth and nonincreasing function equal to 1 on $(-\infty, -1]$ and to 0 on $[0,+\infty)$, and $\gamma>0$ a positive constant such that $\gamma\min_{-1\leq s\leq0} (-v'(s))\geq ||\varphi''||_\infty$. Clearly $\overline{v}$ is a supersolution of \eqref{pb:homomono} for all $t\geq 0$ and $x\geq c^*t+\gamma(1-e^{-\eta t})$. On the other hand for any $t\geq 0$ and $x\leq c^*t+\gamma(1-e^{-\eta t})$, we have that $\overline{v} (t,x) \in (1 - \eta, 1+ \eta)$ and therefore
\begin{align*}
\partial _t \overline{v}-\Delta\overline{v}-g_+(\overline{v})&\geq \eta e^{-\eta t}\left(-\gamma v'-\eta\varphi-c^*\varphi'-\eta\gamma e^{-\eta t}\varphi'-\varphi''-\frac{g_+ '(1)}{2}\varphi\right)\\
&\geq \eta e^{-\eta t}\left(-\gamma v'-\varphi''-\frac{g_+ '(1)}{4}\varphi\right).
\end{align*}
If $x- c^*t-\gamma(1-e^{-\eta t})\leq -1$ it follows from $g_+ '(1)<0$ and $v'<0$ that 
$$\partial_t \overline{v}-\Delta \overline{v}-g_+(\overline{v})\geq0.$$
If $-1\leq x- c^*t-\gamma(1-e^{-\eta t})\leq 0$, the same inequality follows from $\gamma\min_{-1\leq s\leq0} (-v'(s))\geq ||\varphi''||_\infty$. Thus $\overline{v}$ is a supersolution of \eqref{pb:homomono} in $\Omega$, together with the Neumann boundary condition when $\omega$ is bounded.

Let us now check that the solution $u$ lies below $\overline{v}$ after some finite time. First, it follows from Lemma~\ref{pratique} and the fact that $p_c^+ <1$ that
\begin{equation}\label{presque}
\limsup_{ t\to +\infty} \sup_{(x,y) \in \R^{N}}  u (t,x,y) \leq 1.
\end{equation}
Moreover, because we assume that the support of the initial datum $u_0$ is bounded in the right $x$-direction, we can compare the solution with $\overline{v}_2 (t,x):=\frac{e^{At}}{4\pi t}\int_\R \max_{y\in\omega} u_0(z,y) e^{-\frac{|x+ct-z|^2}{4t}}dz$, which is a supersolution of \eqref{problemmoving} as soon as $A \geq  \sup g_+(u)/u$. We get that $u(t,x,y)$ decays faster than any exponential as $x \to +\infty$, in a uniform way with respect to $y \in \omega$. Together with~\eqref{presque}, this implies that one can find $T>0$ and some $X>0$ such that
$$u (T,x,y) \leq \overline{v} (0,x-X),$$
for all $x \in \mathbb{R}$ and $y \in \omega$. Proceeding as above, the solution associated with the initial datum $\overline{v} (0,\cdot)$ converges to~$0$ as time goes to $+\infty$, and by a comparison principle the solution $u$ vanishes.

\section{Sharp threshold phenomena}\label{sec:sharpthresholds}

We are now in a position to prove Theorems~\ref{th:threshold_speed} and~\ref{th:threshold_u0}. Our method relies on the fact that, when the solution does not spread, then it satisfies some uniform in time exponential estimates which allow us to apply Lemma~\ref{lemma:maxprinciple} and Proposition~\ref{othercritical}; see Proposition~\ref{th:boundexp} below.

\subsection{Uniform in time exponential estimates}

Here we prove the following proposition on the asymptotic behaviour of $u(t,x,y)$ as $x \to +\infty$. 
\begin{Prop}\label{th:boundexp}
 Let $u_0$ be as in \eqref{initcond}, i.e. it is bounded, nonnegative and has finite support in the right $x$-direction. Assume that the corresponding solution~$u$ of \eqref{problemmoving} does not spread (in the sense of Definition~\ref{def:dyn}).
\begin{itemize} 
\item If $c\in(2\sqrt{g_+'(0)},c^*)$, then for all $\gamma>0$ such that $c>2\sqrt{g_+ '(0)+\gamma}$, there exist $\varepsilon>0$ and $X_\varepsilon>0$ such that for any $t\in(0,+\infty)$, $y \in \omega$ and $x \geq X_\varepsilon$,
$$u(t,x,y)\leq \varepsilon e^{-\lambda_\gamma (c)(x - X_\varepsilon)},$$
where $\lambda_\gamma (c) :=\frac{c+\sqrt{c^2-4(g_+'(0)+\gamma)}}{2}$.
\item If $c=2\sqrt{g_+ '(0)} < c^*$, then for all $\beta\in (0,1)$, there exist $X_r^1 > X_r^2 >0$ such that for all $t\in(0,+\infty)$, $y \in \omega$ and $x\geq X_r^1$,
$$u(t,x,y)\leq (1+(x - X_r^2)^\beta)e^{-\frac{c}{2} (x- X_r^2) }.$$
\end{itemize}
\end{Prop}
The above proposition only deals with the case when $c \in [2 \sqrt{g_+ ' (0)}, c^*)$. The reason is that other speed ranges have already been dealt with in the previous section. However, one may easily check that Proposition~\ref{th:boundexp} also holds true for any speed $c \geq c^*$.

Notice that the above Proposition~\ref{th:boundexp} is closely related to the energy approach introduced in~\cite{Heinze} and used in several papers to prove convergence to travelling wave solutions~\cite{GR,M04,MN12,MZ,Risler}. We also refer to~\cite{BN} where the energy approach was used in the context of climate change models. Indeed, consider the functional
$$E_c[w] :=\int_{\R \times \omega } e^{cx}\left\{\frac{| \nabla w |^2}{2}-F(z,w)\right\},$$
with $F (x,y,w) = \int_0^w f(x,y,s) ds$, and which is well-defined for any function~$w$ in the weighted Sobolev space $H^1 (\omega \times \R , e^{cx} dx dy)$. The functional $E_c $ is the natural energy associated with~\eqref{problemmoving} in the sense that~\eqref{problemmoving} is a gradient flow generated by~$E_c$.

Then, when $c > 2 \sqrt{g'(0)}$ and $\omega$ is bounded, Proposition~\ref{th:boundexp} implies that the energy $E_c [ u(t,\cdot)]$ of the solution remains bounded uniformly in time. Using the energy functional as a Lyapunov function, one can then prove the large time convergence (at least up to a time sequence) to some stationary state. Such an approach has been used especially in~\cite{MZ} where the authors also obtained sharp dichotomy results by observing that the energy remains bounded as time goes to infinity if and only if spreading does not occur.

Note however that if either $c = 2\sqrt{g'(0)}$ or $\omega = \R^{N-1}$, Proposition~\ref{th:boundexp} does not guarantee the boundedness of the energy, which is why we will use a different approach to prove our main results.

\begin{proof} We begin the proof by noting that there does not exist any sequence of times $t_n$ such that $u(t_n,\cdot)$ converges (locally uniformly) to $p^+_c$. Indeed, if such a sequence of times exists, then one can find some $n$ large enough so that $u(t_n, \cdot)$ lies above the compactly supported subsolution $v^\delta$ (up to some shift to the right, so that it remains a subsolution) constructed in subsection~\ref{sec:interm}. Hence spreading occurs, which contradicts our assumptions. 

Next we prove the following lemma.
\begin{Lemma}\label{uepsilonthm}
Under the assumptions of Proposition \ref{th:boundexp}, for all $\varepsilon>0$, there exists $X_\varepsilon \in \R$ such that 
$$\forall t \geq 0 , \ \forall x \geq X_\varepsilon, \ \forall y \in \omega, \quad u(t,x,y)\leq\varepsilon .$$
\end{Lemma}
\begin{proof} We proceed by contradiction and let $\varepsilon >0$, $t_n>0$ and $(x_n,y_n) \in \Omega$ with $x_n \to +\infty$ be such that $u(t_n,x_n,y_n) > \varepsilon$. 

Moreover, by boundedness of $u$ and hence of its spatial derivatives, we find $\eta >0$ such that for all $n$,
$$u(t_n, x,y) \geq \varepsilon \chi_{B_\eta (x_n,y_n)} (x,y) =: \underline{v}_{0} (x-x_n, y-y_n),$$
where $\chi$ denotes the characteristic function and $B_\eta (x_n,y_n)$ the ball of radius $\eta$ centered at $(x_n,y_n)$. 

As we already did several times, choose $\delta >0$ small and $X_\delta$ large so that, when $x \geq X_\delta$, $y\in\omega$, $f (x,y,\cdot) \geq g_\delta (\cdot)$ with $g_\delta$ a monostable type nonlinearity, whose zeros are 0 and $1- \delta$ and minimal wave speed is $c^*_\delta > c$.  We consider $\underline{v}$ the solution of the monostable equation 
$$\partial_t \underline{v} = \Delta \underline{v} + g_\delta (\underline{v})$$
in $\R^{N}$, 
together with the initial datum $\underline{v}_0 (x,y)$. In the case when~$\omega$ is bounded one may instead define $v_0 = \varepsilon \chi_{\{|x-x_n| \leq \eta\} }$ so that $v$ does not depend on~$y$. As mentioned before, the solution spreads with speed $c^*_\delta$ according to the seminal results of Aronson and Weinberger~\cite{AW}. 
In particular, up to decreasing $\delta$, there exists some $T>0$ such that $\underline{v} (T,\cdot)$ lies above the compactly supported subsolution $v^\delta $ from subsection~\ref{sec:interm}.

Unfortunately we cannot directly compare $u$ with $\underline{v}$, so we instead consider $\underline{v}^D$ the solution of
\begin{equation*}
\left\{
\begin{array}{l}
\partial_t \underline{v}^D = \Delta \underline{v}^D + g_\delta (\underline{v}^D), \vspace{3pt} \\
\underline{v}^D (t=0) = \underline{v}_0 ,
\end{array}
\right.
\end{equation*}
in the truncated domain $\{ x \geq -D\}$ with $D>0$, and a Dirichlet boundary condition $\underline{v}^D (x=D) = 0$. 
Clearly $\underline{v}^D$ converges (locally uniformly in time and space) to $\underline{v}$ as $D \to +\infty$, and thus we can choose $D$ large enough so that $\underline{v}^D (T,\cdot)$ again lies above $v^\delta$.

Furthermore, we have that $\underline{v}^D (x-x_n+ct,y+y_n)$ is a subsolution of~\eqref{problemmoving} as long as $-D + x_n - ct \geq X_\delta$, which includes the interval of time $[0,T]$ provided $n$ is large enough. Applying the comparison principle, we get that
$$u(t_n + t,x,y) \geq v^\delta (x-x_n,y- y_n)$$
for large enough $n$, which implies that $u$ spreads (see also Remark~\ref{rk:yshift}), a contradiction.
\end{proof}
\begin{Rk}
In some sense, the proof of the lemma above completes what we have shown in subsection~\ref{sec:interm}. Indeed, we had shown that if the initial datum is large enough, then spreading occurs. Here we see that spreading also occurs for small initial data, provided their support is far enough to the right. This is to be expected considering that $c$ is less than the spreading speed of the limiting equation as $x \to +\infty$.
\end{Rk}

We go back to the proof of Proposition \ref{th:boundexp}. Let us first assume that $c \in (2 \sqrt{g_+ '(0)},c^*)$, and choose any $\gamma >0$ such that $c > 2 \sqrt{g_+'(0) +\gamma}$. Let then $\varepsilon>0$ be small enough so that $g_+ ' (s)<g_+ '(0)+\gamma$ for all $s\in[0,\varepsilon]$. Define 
$$v(x)=\varepsilon e^{-\lambda_\gamma (c) (x -X_\varepsilon)},$$
where $X_\varepsilon >0 $ comes from Lemma~\ref{uepsilonthm} and $\lambda_{\gamma} (c)$ was defined in Proposition~\ref{th:boundexp}. Noting that $\lambda_\gamma (c)$ is the largest root of $\lambda^2 - c \lambda + g_+'(0) + \gamma$ and from our choice of $\varepsilon$, we get that for all $x \geq X_\varepsilon$,
\begin{align*}
\partial_t v- \partial_x^2 v -c \partial_x v-g_+ (v)&=-\lambda_\gamma (c)^2v+c \lambda_\gamma (c) v-g_+ (v) \geq 0.
\end{align*}
Moroever, up to increasing $X_\varepsilon$ and without loss of generality, we have that $u_0 \equiv 0 \leq v$ in $[X_\varepsilon, + \infty) \times \omega$. Using Lemma~\ref{uepsilonthm}, we also know that 
$u(t, X_\varepsilon,y )\leq \varepsilon = v(X_\varepsilon),$
for all $t \geq 0$ and $y \in \omega$. Therefore, applying the parabolic maximum principle 
we conclude as wanted that 
$$\forall (x,y) \in [X_\varepsilon,+\infty) \times \omega, \ t\geq 0, \quad u(t,x,y)\leq v(x).$$

We now consider the case $c=2\sqrt{g_+ '(0)}$ and define, for all $0< \beta<1$, $x\geq0$,
$$v(x)=(1+x^\beta)e^{-\frac{c}{2}x},$$
Thanks to the $C^{1,r}$ regularity of $g$, one may find $X>0$ large enough so that $v(x)$ again satisfies
$$\partial_t  v - \partial_x^2 v - c\partial_x v - g_+ (v) \geq 0$$
for all $x \geq X$. We refer to the vanishing scenario of subsection~\ref{sec:interm} where the same computation is performed.

On the other hand, since $v(X)>0 $,
it follows from Lemma~\ref{uepsilonthm} that there exists $X' >0 $ such that, for all $t\geq 0$ and $y\in\omega$,
$$u(t,X',y ) \leq v(X) .$$
Up to increasing $X'$ and since $u_0$ has compact support on the right $x$-direction, we can also assume that $u_0 (x) = 0 \leq v (x - X' + X)$ for all $x \in [X' , +\infty)$.
Therefore, we apply the comparison principle again and conclude that for all $t \geq 0$, $x \geq X' $ and $y\in\omega$,
$$u(t,x,y) \leq v (x - X' + X) =  \left((1 + \sqrt{x-X' + X} \right) e^{-\frac{c}{2} (x - X' +X)}.$$
Letting $X_r^1= X'$ and $X_r^2 = X' - X$, we reach the wanted inequality and the proposition is proved.
\end{proof}

\subsection{Sharp threshold for the speed}\label{sec:sharpspeed}

Let us now turn to the existence of a threshold speed when the initial datum is fixed. In this section, we will make a slight change in our notations and define $u^c (t,x,y)$ the solution of \eqref{problemmoving}, in order to highlight the dependence on the speed parameter~$c$.

We already observe that, thanks to the monotonicity of $f(x,y,s)$ with respect to $x$, for any $c_1 < c_2$ the function $u^{c_1} (t,x+ (c_2-c_1) t,y)$ is a supersolution of \eqref{problemmoving} with $c=c_2$. Therefore, by comparison, we get that
\begin{equation}\label{speed_monotone}
u^{c_2} (t,x,y) \leq u^{c_1} (t,x + (c_2 - c_1)t ,y)
\end{equation}
for all $t \geq 0$ and $(x,y) \in \Omega$. It immediately follows that, if $u^{c_1} (t,x,y)$ vanishes in the sense of Definition~\ref{def:dyn}, i.e. it converges uniformly to~0 as $t\to +\infty$, then $u^{c_2} (t,x,y)$ also vanishes. Similarly, if $u^{c_1} (t,x,y)$ is grounded, or more generally if it decays to 0 as $x \to +\infty$ uniformly with respect to $y \in \omega$, then the above inequality implies that $u^{c_2} (t,z)$ converges to 0 uniformly in any set $[X,+\infty) \times \omega$ with $X \in \mathbb{R}$. Thanks to Lemma~\ref{pratique} and since $p_{c_2}^+ (-\infty,\cdot) = 0$, one can actually show that the convergence to 0 is uniform in space, i.e. vanishing occurs at the forced speed $c_2$.

Next we define
\begin{eqnarray*}
\overline{c}  &:=& \max \{ c \geq 0 \, | \ \forall  0 \leq c' < c , \ u^{c'} (t,x,y) \to p^+_{c'} (x) \mbox{ as } t \to +\infty \mbox{ locally uniformly, i.e. spreading occurs} \}, \vspace{3pt}\\
\underline{c} & :=& \min \{ c \geq 0 \, | \ \forall c' > c , \ u^{c'} (t,x,y) \to 0 \mbox{ as } t \to +\infty \mbox{ uniformly, i.e. vanishing occurs} \}.
\end{eqnarray*}
From Section~\ref{sec:3speed}, it is clear that both $\overline{c}$ and $\underline{c}$ are well defined and belong to the interval $[2 \sqrt{g'(0)},c^*]$. Moreover, it follows from the discussion above on the monotonicity with respect to $c$, and the fact that the solution may either spread or satisfy the exponential bounds of Proposition~\ref{th:boundexp}, that $\underline{c} = \overline{c}$. From now on, we denote this speed by $c (u_0)$. We have already proved Theorem \ref{th:threshold_speed}$(i)$.\\

It only remains to investigate how the solution of \eqref{problemmoving} behaves when $c = c (u_0)$, that is to prove Theorem~\ref{th:threshold_speed}$(ii)$. We will first show that, under the assumptions of Theorem~\ref{th:threshold_speed}$(ii)$, if the solution does not vanish then it remains positive in large time locally in space. This leads to the following proposition:
\begin{Prop}\label{prop:behavInter}
Under the assumptions of Theorem \ref{th:threshold_speed}$(ii)$ and if $c \in [2 \sqrt{g_+'(0)},c^*)$, then the solution~$u^c$ of \eqref{problemmoving}-\eqref{initcond} either spreads, vanishes, or is grounded in the sense of Definition~\ref{def:dyn}.
\end{Prop}
In order to prove Proposition \ref{prop:behavInter} when $\omega=\R^{N-1}$, we will need the following lemma about the behaviour of the solution of Problem \eqref{problemmoving} when $|y|$ is large.
\begin{Lemma}\label{reflec}
Under the assumptions of Theorem \ref{th:threshold_speed}$(ii-b)$, i.e. the reaction term $f$ does not depend on $y$ and
$$support (u_0) \subset [-X,X] \times [-Y,Y]^{N-1},$$
then for all $i\in\{1,\dots,N-1\}$, $t >0$ and $(x,y)=(x,y_1,\dots,y_{N-1})\in\R^N$, we have
$$\begin{array}{lcl}
y_i \geq Y & \Rightarrow & \partial_{y_i} u^c < 0, \vspace{5pt}\\
y_i \leq -Y & \Rightarrow  & \partial_{y_i} u^c > 0.
\end{array}$$
\end{Lemma}
\begin{proof}
This follows from a standard reflection method, as it is done for instance in~\cite{DM}. 
Consider any $y^0_i<-Y$ and define 
$$w_i(t,x,y)=u^c (t,x,y)-u^c (t,x,y_1,\dots,2y_i^0-y_i,\dots,y_{N-1}).$$
Then one can check that $w_i$ satisfies
\be\begin{cases}
\partial_t w_i-\Delta w_i-c\partial_x w_i=b(t,x,y)w_i, &\forall \: t>0,\: (x,y)\in\R^N,\: y_i<y_i^0,\\
w_i(t,x,y)=0, &\forall \: t>0,\: (x,y)\in\R^N,\: y_i=y_i^0,\\
w_i(0,x,y)\leq \not \equiv 0, &\forall \: (x,y)\in\R^N,\: y_i<y_i^0,
\end{cases}\ee
where the function~$b$ is bounded, using Hypothesis~\ref{hypsPb}(2) and the assumptions of Theorem \ref{th:threshold_speed}$(ii-b)$.
Then by the strong maximum principle we know that $w_i(t,x,y)<0$ for all $t>0,\: (x,y)\in\R^N,\: y_i<y_i^0$ and using the Hopf lemma we also know that 
$$\partial_{y_i} w(t,x,y_1,\dots,y_i^0,\dots,y_{N-1})=2\partial_{y_i} u^c (t,x,y_1,\dots,y_i^0,\dots,y_{N-1})>0.$$
The case $y_i^0>Y$ can be argued similarly.
\end{proof}
Now let us prove Proposition \ref{prop:behavInter}.
\begin{proof}
Let us assume that spreading does not occur. In particular, the exponential bounds of Proposition~\ref{th:boundexp} hold. We will prove that, if there exist $(x_0,y_0) \in \Omega$ and a sequence $t_n \to +\infty$ such that $u(t_n,x_0,y_0) \to 0$, then the solution vanishes. More precisely, we will show that there exists $n$ large enough such that
\begin{equation}\label{eq:m2}
u^c (t_n,\cdot) \leq \overline{u}_0
\end{equation}
in $\Omega$, where $\overline{u}_0$ denotes the supersolution constructed in the vanishing scenario of subsection~\ref{sec:interm}. 

Recalling Lemma~\ref{pratique}, the supremum limit of $u$ as time goes to infinity is less than the maximal steady state $p_c^+$. Since $p_c^+ (-\infty,\cdot) =0$, one can find $X_- >0$ such that $u^c(t,x,y) \leq \overline{u}_0 (x)$ for all $x \leq X_-$ and $y \in \omega$. Moreover, thanks to Proposition~\ref{th:boundexp} and by construction of $\overline{u}_0$ in subsection~\ref{sec:interm}, there also exists $X_+$ large enough such that, for all $t \geq 0$, $x \geq X_+$ and $y \in \omega$,
$$u^c (t,x,y) \leq \overline{u}_0 (x).$$
Note that when $c= 2\sqrt{g'_+ (0)}$, this requires to apply Proposition~\ref{th:boundexp} with some $\beta'$ chosen to be smaller than the $\beta$ involved in the construction of $\overline{u}_0$.

Next, it follows from our choice of the sequence $t_n$ that $u^c(t_n, \cdot)$ converges locally uniformly to 0 as $n \to +\infty$. In particular, for any $Y>0$, we can find $n$ large enough such that
$$\sup_{x \in [X_-,X_+], y \in \omega \cap [-Y,Y]^{N-1}} u^c (t_n,x,y) \leq \overline{u}(x).$$
Using Lemma \ref{reflec} in the case $\omega=\R^{N-1}$ , we obtain that
$$\sup_{x \in [X_-,X_+], y\in\omega} u^c (t_n,x,y) \leq \overline{u}(x),$$
for large enough~$n$. Together with previous inequalities, that means that \eqref{eq:m2} holds in the whole domain~$\Omega$. By comparison, the solution vanishes.

We have just proved that if the solution neither spreads nor vanishes, then for any $(x,y) \in \Omega$, one has $\liminf_{t \to +\infty} u^c (t,x,y) >0$. Grounding in the sense of Definition~\ref{def:dyn} immediately follows from another application of Proposition~\ref{th:boundexp}.
\end{proof}

Now we use Proposition~\ref{prop:behavInter} to investigate more precisely the behaviour of the solution at the threshold speed, that is when $c=c(u_0)$. 
Let us first prove the following claim. 
\begin{claim}\label{claim:notspreadspeed} When $c = c(u_0)$, then spreading does not occur.\end{claim} 
Proceed by contradiction and assume that $u^{c (u_0)} (t,x,y)$ the solution of \eqref{problemmoving} with $c=c (u_0)$ converges locally uniformly to $p^+_{c (u_0)} $.
As we already dealt with the case $c=c^*$ where vanishing necessarily occurs, we only need to consider $c \in [2 \sqrt{g'(0)},c^*)$. Using the fact that $p^+_{c (u_0)} (x,y) \to 1$ as $x \to +\infty$, uniformly in $y\in\omega$, then for any $R>0$ and $\varepsilon >0$, one can find some $X>0$ and $T>0$ large enough such that $$u^{c (u_0)} (T,x,y) \geq (1 -\varepsilon) \chi_{B_R (0)} (x-X,y),$$ 
in $\mathbb{R} \times \omega$. 
By standard estimates, one can show that the solution of \eqref{problemmoving} depends continuously, in the locally uniform topology, on the parameter $c$.
In particular, for $\varepsilon ' >0$ small enough and any $c \in [c (u_0), c(u_0) + \varepsilon']$, the solution $u^{c}$ of \eqref{problemmoving} also satisfies $u^{c} (T,x,y) \geq (1 -2 \varepsilon) \chi_{B_R (0)} (x-X,y)$ for all $x \in \mathbb{R}$ and $y \in \omega$.

Now choose $R>0$ and $\varepsilon >0$ so that $(1-2\varepsilon) \chi_{B_R (0)} (x ,y) \geq v^\delta (x,y)$, where $v^\delta$ is a subsolution of~\eqref{problemmoving} with $c = \frac{c(u_0) + c^*}{2}  \in ( c(u_0),c^*)$, as we constructed in subsection~\ref{sec:interm} so that the associated solution spreads. According to the discussion above on the monotonicity of solutions with respect to~$c$, as well as the monotonicity of $f$ with respect to $x$, it is clear that the solution of \eqref{problemmoving} with initial datum $v^\delta (\cdot - X, \cdot)$ also spreads for any $c \in  (c (u_0) , \frac{c(u_0) + c^*}{2}   )$. Since
$$u^c (T,x,y) \geq (1-2\varepsilon) \chi_{B_R (0)} (x - X,y)  \geq v^\delta (x-X,y),$$
we conclude by comparison that $u^c$ spreads for all $c \in [c(u_0),c(u_0) + \varepsilon ' ]$. This contradicts the definition of $c(u_0)$ and proves Claim~\ref{claim:notspreadspeed}.\medskip

It remains to show the following claim.
\begin{claim}\label{claim:notvanishspeed} When $c= c(u_0)> 2\sqrt{g_+ '(0)}$, then vanishing does not occur. 
\end{claim}
In particular, it is impossible that $c(u_0)= c^*$ if $c^* > 2 \sqrt{g_+ '(0)}$. To prove this claim, we again proceed by contradiction and assume that $u^{c(u_0)} (t,\cdot)$ converges uniformly to 0 as $t \to +\infty$. We will then show that vanishing also occurs when~$c$ is close to $c(u_0)$, contradicting the definition of $c(u_0)$. 

Recall first that, by the maximum principle, we have for all $c \geq 0$ that
\begin{equation}\label{defM}
\|u^{c} \|_{L^\infty (\R_+ \times \Omega)} \leq \max \{ \|u_0\|_{L^\infty (\Omega)} , 1\} =: M.
\end{equation}
Let $\varepsilon >0$ be small enough so that $ c(u_0) -  3\varepsilon > 2 \sqrt{g_+ '(0)}$ and $\delta>0$ be small enough so that for all $s\in[0,\delta)$,
$$c(u_0)- 3 \varepsilon>2\sqrt{g'_+(s)}.$$
We introduce
$$\overline{u}_0(x):= \min\{ M, \delta e^{-\frac{c(u_0)-\varepsilon}{2}(x+X)} \},$$ 
where $M$ was defined in~\eqref{defM} and $X>0$ is chosen such that for all $x<-X$ and $s \leq M$, $f(x,y,s)\leq -\frac{1}{2}s$. Then one can prove that $\overline{u}_0$ satisfies 
$$\partial_x^2 \overline{u}_0  + (c (u_0) - 2 \varepsilon) \partial_x \overline{u}_0  + f(x,y,\overline{u}_0) \leq 0 ,$$ 
in the sense of a generalised supersolution. As $\overline{u}_0$ is a nonincreasing function, it immediately follows that it is also a supersolution of~\eqref{problemmoving} for any $c' \in [c(u_0) -2 \varepsilon, c(u_0)]$.

Next, let $X' >0$ be such that the support of $u_0$ is included in $(-\infty,X')$, that $\overline{u}_0 (x) \geq M$ for any $x \leq -X'$ as well as, by Proposition~\ref{th:boundexp}, 
$$\forall t \geq 0, \ \forall x \geq X' , y\in\omega, \quad u^{c (u_0)} (t,x,y) \leq \frac{\overline{u}_0 (x)}{2}.$$
Because $u^{c (u_0)}$ goes extinct as $t \to +\infty$, there also exists $T>0$ such that
$$\forall | x | \leq X', y\in\omega , \quad u^{c (u_0)} (T,x,y) \leq \frac{\overline{u}_0 (x)}{2}.$$
Up to reducing $\varepsilon$ and by continuity of solutions of \eqref{problemmoving} with respect to $c$ in the locally uniform topology, we get for any $c' \in [ c (u_0)  - 2 \varepsilon, c (u_0)]$ that, in the case when $\omega$ is bounded,
$$\forall 0 \leq t \leq T, y \in \omega  , \quad u^{c' } (t, X',y) \leq \overline{u}_0 (X'),$$
and 
$$\forall | x | \leq X', y\in\omega , \quad u^{c '} (T,x,y) \leq \overline{u}_0 (x).$$
When $\omega = \R^{N-1}$, we only get that
$$\forall 0 \leq t \leq T, y \in  [-Y,Y]^{N-1},  \quad u^{c' } (t, X',y) \leq \overline{u}_0 (X'),$$
and 
$$\forall | x | \leq X', y\in  [-Y,Y]^{N-1} , \quad u^{c '} (T,x,y) \leq \overline{u}_0 (x).$$
However, thanks to Lemma \ref{reflec}, we again reach the conclusion that those inequalities hold for all $y \in \omega$.

Applying a comparison principle on $(0,T) \times (X',+\infty)\times\omega$, and since $\overline{u}_0 (x) \geq M \geq \|u^{c'} \|_{L^\infty (\R_+ \times \R)}$ for all $x \leq -X'$, one can check that for all $c' \in[c(u_0)- 2\varepsilon,c(u_0)]$,
$$\forall x \in \R,y\in\omega ,\quad u^{c'} (T,x,y) \leq \overline{u}_0 (x).$$
Let $\overline{u}$ be the solution of~\eqref{problemmoving} with $c=c(u_0) -  2\varepsilon$ starting with the supersolution $\overline{u}_0$. Then we have that $u^{c(u_0) -  2\varepsilon} (t+T,x,y) \leq \overline{u} (t,x)$ for all $x \in \R$, $y\in\omega$ and $t \geq 0$.
Moreover, the function $\overline{u}$ is decreasing with respect to time and therefore it converges to a stationary solution as time goes to infinity. Because  $\overline{u}_0(x)=o(e^{-\frac{c(u_0)-2\varepsilon}{2}})$ as $x\to+\infty$, one can proceed in the same way than at the end of section \ref{sec:interm}, to conclude that $\overline{u}$ actually converges (uniformly in space) to 0 as $t \to +\infty$. This means that $u^{c(u_0)-2 \varepsilon}$ vanishes, and so does $u^{c'}$ for any $c'\in[c(u_0)-2\varepsilon,c(u_0)]$ due to~\eqref{speed_monotone}. We have reached a contradiction.\\

Putting together Proposition~\ref{prop:behavInter} and Claims~\ref{claim:notspreadspeed} and~\ref{claim:notvanishspeed}, we conclude as announced that, when $c = c (u_0) > 2 \sqrt{g'(0)}$, then grounding occurs. Theorem~\ref{th:threshold_speed} is proved. \qed

\subsection{Sharp threshold for the initial datum}\label{sec:sharpinit}

In this section we assume that $2\sqrt{g'_+(0)}<c^*$, $c\in[2\sqrt{g'_+(0)},c^*)$ and $\omega$ is bounded. 
We consider a family of initial data $(u_{0,\sigma})_{\sigma >0}$ satisfying the assumptions of Theorem~\ref{th:threshold_u0} and denote by $u_\sigma$ the corresponding solution of Problem \eqref{problemmoving}.

We first define the set $\Sigma_0$, respectively $\Sigma_1$, of all $\sigma$ such that the solution $u_\sigma$ of \eqref{problemmoving} with initial condition $u_{0,\sigma}$ vanishes, respectively spreads. Then because of the monotonicity of the problem with respect to the initial condition $u_{0,\sigma}$, there exist $\underline{\sigma}\leq\overline{\sigma}$ such that 
$$(0,\underline{\sigma})\subset \Sigma_0\subset (0,\underline{\sigma}]$$
and 
$$(\overline{\sigma},+\infty)\subset\Sigma_1\subset[\overline{\sigma},+\infty).$$
We know from Proposition \ref{prop:behavInter}, that the solution either spreads, vanishes or is grounded. Thus if $\underline{\sigma}\neq\overline{\sigma}$ then for all $\sigma\in(\underline{\sigma},\overline{\sigma})$, the solution $u_\sigma$ is grounded, in the sense of Definition \ref{def:dyn}. We will start by the following claim.
\begin{claim}\label{claim:sharpinit}
$\underline{\sigma}\not\in\Sigma_0$ and $\overline{\sigma}\not\in\Sigma_1$.
\end{claim}
We argue by contradiction and first assume that $\overline{\sigma}\in\Sigma_1$. Considering $\eta>0$ small, then we can prove that the solution $u_{\overline{\sigma}-\eta}$ again spreads using the same argument as in the proof of Claim~\ref{claim:notspreadspeed}. This contradicts the fact that $\overline{\sigma}$ is the infimum of $\Sigma_1$, and thus $\Sigma_1=(\overline{\sigma},+\infty)$. We omit the details.

Now assume that $\underline{\sigma}\in\Sigma_0$, and again using similar arguments as before (see the proof of Claim \ref{claim:notvanishspeed}) we can show that $u_{\underline{\sigma}+\eta}$ vanishes for $\eta>0$ small enough, leading to a contradiction. Because this part is slightly different, here we give some details for the sake of clarity. We first deal with the case $c>2\sqrt{g'_+(0)}$. Choose $\delta >0$ and $\varepsilon >0$ small enough so that $c>2\sqrt{g'_+(s)+\varepsilon^2}$ for all $s\in[0,\delta]$, as well as $X >0$ large enough so that $f(x,y,s)\leq -\frac{s}{2}$ for all $x<-X$ and $s \leq \|u_{\underline{\sigma}+1}\|_\infty$. Then one can prove that the function
$$\overline{u}_0 (x):= \min \{ \|u_{\underline{\sigma}+1}\|_\infty, \delta e^{-\frac{c+\varepsilon}{2}(x+X)} \}$$
is a supersolution of \eqref{problemmoving}. 
Then the proof proceeds as in Claim~\ref{claim:notvanishspeed}. Using Proposition~\ref{th:boundexp}, the fact that $u_{\underline{\sigma}}$ vanishes and a comparison principle argument on a right half-space, one can find that there exists some $T>0$ such that, for all $\eta \in (0,1)$ small enough,
$$u_\sigma(T,x,y)\leq \overline{u}_0(x).$$
Proceeding as in the end of section \ref{sec:interm}, it follows from the fact that $\overline{u}_0(x)=o(e^{-\frac{c}{2}x})$ as $x\to+\infty$ that the solution of \eqref{problemmoving} with the initial datum $\overline{u}_0$ vanishes, and thus so does $u_\sigma$. We reached a contradiction.

Next, consider the case when $c=2\sqrt{g'_+(0)}<c^*$. Define 
$$\tilde{u}_0(x):= M(1+(x+X)^\beta)e^{-\frac{c}{2} (x+X)}$$
where $M>0$, $\beta \in (0,\frac{1}{2})$ and $X>0$. Define also $X'>0$ such that the support of $u_{0,\underline{\sigma} +1}$ is included in $(-X',X')$, and such that for all $x \leq -X'$ and $y \in \omega$, one has $f(x,y,s) <0$. By the same computations as in the vanishing scenario of section~\ref{sec:interm}, we can choose $X$ large enough (depending on $M$, $\beta$ and $g_+$) such that
$$\overline{u}_0 (x) := \left\{
\begin{array}{ll}
\tilde{u}_0 (-X') & \mbox{ if } x \leq - X' , \vspace{3pt}\\
\tilde{u}_0 (x) & \mbox{ if } x > X',
\end{array}
\right.
$$
is a generalised supersolution of \eqref{problemmoving}.

The proof then proceeds similarly as in Claim~\ref{claim:notvanishspeed}. There exists $T>0$ and $\eta >0$ such that for all $\sigma \in (\underline{\sigma},\underline{\sigma}+\eta)$, 
$$\forall \: |x|< X', \ y\in \omega, \quad u_{\sigma} (T,x,y) \leq \overline{u}_0 (x),$$
as well as
$$\forall \: t\in (0,T), \ y \in \omega, \quad  u_\sigma (t,\pm X',y) \leq \overline{u}_0 (X').$$
Applying a parabolic comparison principle in $(0,T) \times (X',+\infty) \times \omega$, as well as in $(0,T) \times (-\infty,X') \times \omega$ (which was not necessary in the proof of Claim~\ref{claim:notvanishspeed}), it is straightforward that the inequality $u_\sigma (T,\cdot) \leq \overline{u}_0$ holds in the whole spatial domain. Due to the fact that $\overline{u}_0(x) = o((1+\sqrt{x})e^{-\frac{c}{2}x })$ as $x \to +\infty$, one can again infer than vanishing occurs for all $\sigma < \underline{\sigma}+\eta$, a contradiction.\\

Claim~\ref{claim:sharpinit} is proved. We now know that $\Sigma_0 = (0,\underline{\sigma})$ and $\Sigma_1 = (\overline{\sigma},+\infty)$, and therefore (remember Proposition~\ref{prop:behavInter}) for all $\sigma \in [\underline{\sigma},\overline{\sigma}]$ the solution $u_\sigma$ is grounded in the sense of Definition~\ref{def:dyn} . It only remains to show that $\underline{\sigma} = \overline{\sigma}$. This immediately follows from the next claim, which will thus conclude the proof of Theorem~\ref{th:threshold_u0}:
\begin{claim}\label{last}
Let $0< \sigma_1 \leq \sigma_2< +\infty$ such that $u_{\sigma_1}$ and $u_{\sigma_2}$ are grounded. Then $\sigma_1 = \sigma_2$.
\end{claim}
Let us first note that, when $c > 2 \sqrt{g_+ '(0)}$, then Proposition~\ref{th:boundexp} implies that both solutions $u_{\sigma_1}$ and $u_{\sigma_2}$ have bounded energy. By a classical argument, this guarantees that there exists a time sequence $t_n \to +\infty$ such that $u_{\sigma_1} (t_n,\cdot)$ and $u_{\sigma_2} (t_n,\cdot)$ converge to some positive stationary states $p_1$ and $p_2$. Moreover, because $u_{0,\sigma_1} \leq u_{0,\sigma_2}$, we have by the comparison principle that $p_1 \leq p_2$. According to Proposition~\ref{othercritical} and the strong maximum principle, we get that $p_1 \equiv p_2$. Now assume by contradiction that $\sigma_1 < \sigma_2$. Then, the hypotheses of Theorem~\ref{th:threshold_u0} imply that there exists some small $\zeta >0$ such that $u_{0,\sigma_1} (x-\zeta,y) < u_{0,\sigma_2 }(x,y)$ for all $x \in \mathbb{R}$ and $y \in \omega$. Because $u_{\sigma_1} (\cdot, \cdot - \zeta,\cdot)$ is a subsolution of \eqref{problemmoving}, it follows that $p_1 (\cdot - \zeta, \cdot) \leq p_2 (\cdot) \equiv p_1 (\cdot)$. This is clearly impossible since $p_1$ is positive and tends to 0 as $x \to +\infty$, thus $\sigma_1 = \sigma_2$.

Unfortunately, this argument fails in the critical case $c= 2\sqrt{g_+ '(0)}$ because Proposition~\ref{th:boundexp} is not sufficient to bound the energy of the solution uniformly in time. Therefore, we are not able to prove that the solution converges (even along a time sequence) to a stationary state and cannot apply Proposition~\ref{othercritical}. However, we will use a similar argument which relies on the fact that Proposition~\ref{othercritical} can be extended in some sense to entire solutions of the evolution equation \eqref{problemmoving}.

We again argue by contradiction and consider $\sigma_1<\sigma_2$ such that grounding occurs. By parabolic estimates, one can find a sequence $t_n\to+\infty$ as $n\to+\infty$ such that $u_{\sigma_1}(t+t_n,x,y)$ and $u_{\sigma_2}(t+t_n,x,y)$ converge respectively to $u_{\sigma_1}^\infty(t,x,y)$ and $u_{\sigma_2}^\infty(t,x,y)$ as $n \to +\infty$, locally uniformly with respect to $t$, $x$ and~$y$. Moreover, $u_{\sigma_1}^\infty$ and $u_{\sigma_2}^\infty$ are positive and bounded entire solutions of~\eqref{problemmoving}, which also satisfy
\begin{equation}\label{endd1}
u_{\sigma_1}^\infty (t,x,y) , u_{\sigma_2}^\infty (t,x,y) \to 0 \quad \mbox{ as } x \to +\infty \  \mbox{uniformly in $t \in \mathbb{R}$ and $y \in \omega$}
\end{equation}
and 
\begin{equation}\label{endd2}
\inf_{t \in \mathbb{R} , y \in \omega} u_{\sigma_1}^\infty (t,0,y), u_{\sigma_2}^\infty (t,0,y) >0.
\end{equation}
Using the assumptions of Theorem \ref{th:threshold_u0}, there exists $\overline{\zeta}>0$ such that for all $\zeta\in[0,\overline{\zeta})$, $u_{0,\sigma_1}(x-\zeta,y)\leq u_{0,\sigma_2}(x,y)$ for all $x\in\R$, $y\in\omega$. 
Due to the monotonicity of $f$ with respect to $x$ we know that $u_{\sigma_1}(\cdot, \cdot-\zeta,\cdot)$ is a subsolution of \eqref{problemmoving} for all $\zeta \geq 0$.  Therefore, for all $\zeta \in [0,\overline{\zeta})$, we get that $u_{\sigma_1}^\infty (\cdot, \cdot - \zeta, \cdot) \leq u_{\sigma_2}^\infty (\cdot)$. We can define (as in the proof of Proposition \ref{othercritical}), 
$$\zeta^*:= \max\{\zeta>0\:|\: u_{\sigma_1}^\infty(\cdot,\cdot-\zeta,\cdot)\leq u_{\sigma_2}^\infty(\cdot)\} \geq \overline{\zeta} >0.$$
By continuity, we have that $u_{\sigma_1}^\infty (\cdot,\cdot - \zeta^* ,\cdot) \leq u_{\sigma_2}^\infty (\cdot)$ and $\inf u_{\sigma_2}^\infty (\cdot) - u_{\sigma_1}^\infty (\cdot, \cdot - \zeta^*,\cdot) =0$. In particular, there exists a sequence $(t_n,x_n,y_n)$ such that $u_{\sigma_1}^\infty(t_n,x_n-\zeta^*,y_n)-u_{\sigma_2}^\infty(t_n,x_n,y_n)\to 0$ as $n\to+\infty$.

Since we assume $\omega$ to be bounded here, we have without loss of generality that $y_n \to y_\infty \in \overline{\omega}$ as $n \to +\infty$. Consider first the case when the sequence $(t_n,x_n)$ is also bounded. Then there exists some point $(t_\infty,x_\infty,y_\infty)$ where $u_{\sigma_1}^\infty (t_\infty, x_\infty - \zeta^*,y_\infty) = u_{\sigma_2}^\infty (t_\infty,x_\infty,y_\infty)$, and the strong maximum principle together with Hopf lemma imply that $u_{\sigma_1}^\infty(\cdot, \cdot-\zeta^*,\cdot)\equiv u_{\sigma_2}^\infty(\cdot)$. Putting this together with~\eqref{endd1} and~\eqref{endd2} which show that $u_{\sigma_1}^\infty$ non trivially depends on~$x$, this contradicts the fact that $\zeta^* >0$.

The case when $x_n$ is bounded but $|t_n| \to +\infty$ can be dealt with similarly. Indeed, one can extract a subsequence such that $u_{\sigma_1}^\infty (t+t_n,x-\zeta^*,y)$ and $u_{\sigma_2}^\infty (t+t_n,x,y)$ converge to some entire solutions of \eqref{problemmoving} which again satisfy~\eqref{endd1} and~\eqref{endd2}. With some slight abuse of notations, one may still denote those limits by $u_{\sigma_1}^\infty (\cdot,\cdot -\zeta^*,\cdot)$ and $u_{\sigma_2}^\infty$, and repeat the argument above to reach a contradiction.

It remains to consider the situation when $|x_n|\to+\infty$. From the above, we know that for any $X>0$, 
$$\inf_{t\in\R,\:x\in[-X,X],\:y\in\omega} u_{\sigma_1}^\infty(t,x-\zeta^*,y)-u_{\sigma_2}^\infty(t,x,y)>0. $$ As in the proof of Proposition \ref{othercritical}, choose $\varepsilon>0$ small enough such that 
$$\inf_{t\in\R,\:x\in[-X,X],\:y\in\omega} u_{\sigma_1}^\infty(t,x-\zeta^*-\varepsilon,y)-u_{\sigma_2}^\infty(t,x,y)>0.$$
By Proposition~\ref{th:boundexp}, we have that $u_{\sigma_2}^\infty$ satisfies the appropriate exponential bound in order to apply the parabolic maximum principle Lemma~\ref{lemma:paramaxprinciple}. Therefore, we get that $u_{\sigma_1}^\infty(t,x-\zeta^* -\varepsilon,y)\leq u_{\sigma_2}^\infty(t,x,y)$ for all $t\in\R$, $x\geq X$ and $y\in\omega$. Moreover, up to increasing $X$ so that $f(x,y,s) \leq -\frac{s}{2}$ for all $x \leq - X$, $y \in \omega$ and $s \leq \|u_{\sigma_2}\|_\infty$, one can apply a classical maximum principle to conclude first that $u_{\sigma_2}^\infty (t,-\infty,y) =0$, and then that $u_{\sigma_1}^\infty(t,x-\zeta^* -\varepsilon,y)\leq u_{\sigma_2}^\infty(t,x,y)$ also for all $x\leq -X$, $t\in\R$ and $y\in\omega$. We have shown that $u_{\sigma_1}^\infty (\cdot, \cdot - \zeta^* - \varepsilon, \cdot) - u_{\sigma_2}^\infty (\cdot)$ in the whole domain $\R \times \R \times \omega$. This contradicts the definition of $\zeta^*$ and ends the proof of Claim~\ref{last}, as well as of Theorem~\ref{th:threshold_u0}.

\section{The case of exponentially decaying initial data}\label{sec:last}

In this last section, we consider an initial datum $u_0$ such that
$$u_0 (x,\cdot) \sim A e^{-\alpha x}$$
as $x \to +\infty$, where $A >0$ and $\alpha >0$.

As one may expect, Theorem~\ref{th:last} mostly follows from comparisons with the solution~$v$ of the homogeneous monostable equation
$$\partial_t v - \Delta v - g_+ (v) =0,$$
which we pose in $\mathbb{R} \times \omega$ (with Neumann boundary condition if necessary) and supplement with the initial datum $u_0$. Note though that Theorem~\ref{th:last} does not address the sharpness of the thresholds between extinction and spreading, which is why here we can avoid most of the technical difficulties of previous sections. Moreover, we will only sketch the arguments, which indeed share similarities to those we used in the previous sections.

It is well known that, if $\alpha \leq \alpha^* := \frac{c^* - \sqrt{c^* - 4 g'_+ (0)}}{2}$ (resp. $\alpha > \alpha^*$) then $v$ spreads with speed $c_\alpha= \frac{\alpha^2 + g_+ '(0)}{\alpha}$ (resp. $c^*$). We refer for instance to \cite{rothe,uchi78}. In particular we have that
$$\limsup_{t \to +\infty} \sup_{x \geq ct,y \in \omega } v (t,x,y) = 0$$
for any $c > c_\alpha$ (resp. $c > c^*$). In the case when $\alpha < \alpha^*$, this can be understood from the fact that the travelling wave with speed $c_\alpha$ also decays as the exponential function $e^{-\alpha x}$ as $x \to +\infty$, and thus it can be used as a supersolution. Note that in the literature, it is often assumed that $u_0 <1$, but it is standard to remove this assumption as we did for instance in Section~\ref{sec:high}. Since $v (x+ct,y)$ is a supersolution of \eqref{problemmoving} and using also Lemma~\ref{pratique}, we infer that vanishing occurs when either $\alpha  \leq \alpha^*$ and $c > c_\alpha$, or $\alpha > \alpha^*$ and $c > c^*$.\medskip

Next, let us note that if $c < 2 \sqrt{g'_+ (0)}$, then spreading occurs as an immediate consequence of Theorem~\ref{th:regime}. We now show that, if $\alpha < \sqrt{g_+ '(0)}$ and $c \in [2 \sqrt{g'_+ (0)}, c_\alpha)$, then the solution also spreads.

Because $f$ converges to $g_+$ as $x \to +\infty$ in the $C^{1,r}$-topology, we can find a monostable function $g_\delta$ such that $g_\delta ' (0)$ is smaller than but arbitrarily close to $g'_+ (0)$, and $f (x,y,s) \geq g_\delta (s)$ for all $x \geq X_\delta$, $y\in \omega$ and $s \geq 0$. In particular, $c > 2 \sqrt{g'_\delta (0)}$ and we can choose $\varepsilon >0$ small enough so that the smallest solution~$\alpha_1 (c')$ of
$$\alpha_1 (c')^2 - c' \alpha_1  (c')+ g'_\delta (0) = 0$$
satisfies $\alpha_1 (c') > \alpha$ for any $c' \in [c,c+\varepsilon)$. Moreover, one can check that
$$\underline{u} (t,x) = \max \{0, a e^{-\alpha_1(c') (x-X_\delta + (c-c')t)} - b e^{-(\alpha_1 (c') + \eta) (x-X_\delta + (c-c')t)} \}$$
is a subsolution of \eqref{problemmoving}, where $\eta >0$, $a>0$ and $b>0$ have to be chosen small enough but independently of $c' \in [c, c+\varepsilon)$. In particular, we can assume without loss of generality that $u_0 \geq \underline{u} (t=0)$. Thus, one gets by comparison that the solution~$u$ of \eqref{problemmoving} with initial datum $u_0$ not only does not vanish, but also satisfies that
$$\liminf_{x \to +\infty} \inf_{y \in \omega} \lim_{t \to +\infty} u (t,x,y) > 0.$$
From this, we infer that spreading occurs. Indeed, the solution of \eqref{problemmoving} with initial datum $\underline{u} (t=0)$ is increasing in time, and since it is also bounded it must converge (locally uniformly) to a stationary state~$p$. According to Proposition~\ref{otherstates}, we have that $p \equiv p_c^+$. By the comparison principle, it easily follows that~$u$ also converges to~$p_c^+$.\medskip

We now investigate the intermediate speed ranges where both spreading and vanishing occur. First, it follows from Theorem~\ref{th:regime} that for any $c < c^*$ and $\alpha >0$, one can find an initial datum $u_0$ behaving as $Ae^{-\alpha x}$ as $x \to +\infty$ such that spreading occurs. It remains to show that, when either $\alpha \in (\alpha^*, \sqrt{g'_+ (0)})$ and $c \in (c_\alpha, c^*)$, or $\alpha \geq \sqrt{g'_+ (0)}$ and $c \in (2 \sqrt{g'_+ (0)},c^*)$, one can also find an exponentially decaying initial datum~$u_0$ so that the solution vanishes.

The argument is similar to the one in section~\ref{sec:interm}, hence we only sketch it. Consider the former case and let $c' \in (c_\alpha, c)$. Then define
$$\overline{u} (x) = \left\{
\begin{array}{ll}
e^{-\alpha X} & \mbox{ if } x \leq - X,\\
 e^{-\alpha (x+ 2X)} & \mbox{ if } x \geq - X.\\
\end{array}
\right.
$$
One can check that, if $X$ is chosen large enough, then $\overline{u} (x+ (c-c')t)$ is a supersolution of~\eqref{problemmoving}. Together with Lemma~\ref{pratique}, this implies that the corresponding solution vanishes. In the latter case $\alpha \geq \sqrt{g'_+(0)}$, defining instead
$$\overline{u} (x) = \left\{
\begin{array}{ll}
e^{-\sqrt{g'_+(0)} X} & \mbox{ if } x \leq - X,\\
 e^{-\sqrt{g'_+(0)} (x+ 2X)} & \mbox{ if } x \geq - X,\\
\end{array}
\right.
$$
and repeating the same argument also lead to a vanishing scenario for any $c > \sqrt{g'_+ (0)}$.

All the situations in Theorem~\ref{th:last} have been dealt with and the proof is finally complete.

\appendix
\section{Maximum principle lemmas}\label{A:maxpple}

\subsection{An elliptic maximum principle}

In our proofs, we use extensively the following lemma:
\begin{Lemma}\label{lemma:maxprinciple}
Assume that $c \geq 2\sqrt{g_+ '(0)}$. Let $\underline{\phi} (x,y)$ and $\overline{\phi}(x,y)$ be two nonnegative functions, respectively a sub and a supersolution of~\eqref{problemmoving}, and assume that they both decay to 0 at $x \to +\infty$ as
$$o(e^{-\frac{c}{2} x})\quad  \left( \mbox{if }c > 2 \sqrt{g_+ '(0)} \, \right),$$
$$o((1+\sqrt{x}) e^{-\frac{c}{2} x})\quad  \left( \mbox{if }c = 2 \sqrt{g_+ '(0)} \, \right),$$
where these estimates are understood to be uniform with respect to $y \in \omega$.

Then there exists $X>0$ large enough such that, for any $X ' \geq X$, if $\underline{\phi} (X',\cdot) \leq \overline{\phi} (X',\cdot)$, then either $\underline{\phi} \equiv \overline{\phi}$ or $\underline{\phi} (x,y) < \overline{\phi} (x,y)$ for all $(x,y) \in (X' ,+\infty) \times \omega$.
\end{Lemma}
Note first that when $c < 2 \sqrt{g_+'(0)}$, then we prove in Section~\ref{sec:3speed} that 0 is unstable with respect to compact perturbations so that there does not exist any supersolution decaying to 0 at infinity. Therefore, it is meaningful to only consider the case $c \geq 2 \sqrt{g_+ '(0)}$.

Moreover, if instead of our assumptions, one chooses $f$ to have a negative derivative with respect to its third variable, then the lemma reduces to the usual maximum principle. In particular, one can clearly derive a similar property on the far left of our problem. Because $\partial_s f(x,y,0) >0$ for large positive $x$ and all $y \in \omega$, the situation is rather different on the far right. However, it is still quite standard that the maximum principle remains valid in an appropriate subspace of fast decaying functions, in which 0 is again stable. This is exactly Lemma~\ref{lemma:maxprinciple}, whose proof we include below for the sake of completeness.

\begin{proof}
Let us first assume that $c > 2 \sqrt{g_+'(0)}$. Let $\varepsilon>0$ be small so that $c > 2 \sqrt{g_+'(0) + 2 \varepsilon}$. We consider $\delta>0$ and $X>0$ such that, for all $x \geq X$, $y\in\omega$ and $s \in [0,\delta]$, one has
$$\partial_s f(x,y,s) \leq g_+'(0) + \varepsilon ,$$
and for all $x \geq X$ and $y \in \omega$,
$$\underline{\phi} (x,y), \, \overline{\phi} (x,y) \leq \delta.$$
Then $\psi := e^{\frac{c}{2}x } \left( \underline{\phi} - \overline{\phi}\right)$ satisfies, when $x \geq X$:
\begin{equation*}
 \Delta \psi - \frac{c^2}{4} \psi + e^{\frac{c}{2} x} \left[ f (x,y,\underline{\phi}) - f (x, y,\overline{\phi})  \right] \geq 0.
\end{equation*}
Proceed by contradiction and assume that $\psi >0$ at some point $(x_0,y_0)$ with $x_0 > X'\geq X$. From our choice of $X$, we have wherever $\psi >0$ that
$$f(x,y,\underline{\phi}) - f (x,y,\overline{\phi}) \leq (g_+ ' (0) + \varepsilon) (\underline{\phi} - \overline{\phi}) < \left(\frac{c^2}{4} - \varepsilon \right) e^{-\frac{c}{2} x} \psi ,$$
hence $\Delta \psi > \varepsilon \psi$. Noting that $\lim_{x \to +\infty} \psi (x,y) =0$ (uniformly with respect to $y$), we can assume without loss of generality that 
$$\psi (x_0,y_0) = \max_{x \geq X', y \in \omega} \psi (x,y),$$
and thus $\Delta \psi (x_0,y_0) \leq 0$, a contradiction. When $\omega = \R^{N-1}$, this may require passing to the limit in a sequence of appropriate shifts of $\psi$. We conclude that $\underline{\phi} \leq \overline{\phi}$ in $(X',+\infty) \times \omega$. By the strong maximum principle, either $\underline{\phi} \equiv \overline{\phi}$ or this inequality is also strict.\medskip

The case $c = 2 \sqrt{g_+ '(0)}$ can be dealt with similarly, by defining instead
$$\psi := \frac{e^{\frac{c}{2} x}}{1+\sqrt{x}} (\underline{\phi} - \overline{\phi}),$$
which satisfies
$$\Delta \psi  + \frac{1}{\sqrt{x}(1+\sqrt{x})} \partial_x  \psi + \left( \frac{1}{4x^{3/2}}- \frac{3}{4x}  \right)\frac{1}{(1+\sqrt{x})^2} \psi+ \frac{e^{\frac{c}{2} x}}{1+\sqrt{x} } \left( f(x,y,\underline{\phi} ) - f (x,y,\overline{\phi}) - g_+'(0) (\underline{\phi}-\overline{\phi}) \right) \geq 0 $$
for all $ x>0$. Because of the critical choice of $c$, here we use the fact that $\partial_s f(x,y,s)$ converges not only locally uniformly in $s$ to $g_+ ' (s)$ as $x \to +\infty$, but also in the $C^{0,r}$ topology where $0<r<1$. Thus one can find $\delta >0$, $X>0$ and $\eta >0$ such that for all $x \geq X$, $y \in \omega$ and $s \in [0,\delta]$, one has
$$\partial_s f(x,y,s) \leq g_+ ' (0)  + \eta s^{r}.$$
Then, when $\psi >0$ and $x$ large enough, we get
\begin{eqnarray*}
& &   \frac{e^{\frac{c}{2} x}}{1+\sqrt{x} } \left(f(x,y,\underline{\phi} ) - f (x,y,\overline{\phi}) - g_+ '(0) (\underline{\phi}-\overline{\phi}) \right) \vspace{3pt} \\
& = & O\left( \frac{e^{\frac{c}{2} x}}{1+\sqrt{x} }(\underline{\phi}-\overline{\phi})^{1+r}\right)\\
& = & O \left(  (\underline{\phi}-\overline{\phi})^r \psi \right) \vspace{3pt}\\
& = & o \left( \psi (1+\sqrt{x})^re^{-r\frac{c}{2}x}\right)\\
& = &  o\left( \frac{\psi}{x(1+\sqrt{x})^2} \right).
\end{eqnarray*}
Proceeding as above, we assume that there exists $(x_0,y_0)\in\Omega$ such that $x_0>X'\geq X$ and $\psi(x_0,y_0)>0$ as well as $\lim_{x\to+\infty} \psi(x,y)=0$ uniformly with respect to y. Then we apply the maximum principle to reach the conclusion.
\end{proof}
The interest of this lemma is that a solution of \eqref{problemmoving} satisfies such exponential bounds when it does not spread: see Proposition~\ref{th:boundexp}. As we cannot a priori exclude the possibilty that the solution converges to nonstationary solutions as time goes to infinity, we also provide below a parabolic version of this maximum principle.

\subsection{A parabolic maximum principle}

\begin{Lemma}\label{lemma:paramaxprinciple}
Assume $c \geq 2 \sqrt{g_+'(0)}$. Assume that $\underline{u}$ is a global in time subsolution, and $\overline{u}$ a global in time supersolution, satisfying the exponential estimates from Lemma \ref{lemma:maxprinciple} as $x\to+\infty$, uniformly with respect to $t \in \mathbb{R}$ and $y\in \omega$.

Then there exists $X>0$ large enough such that, for any $X' \geq X$, if $u_1 (\cdot,X',\cdot) \leq u_2 (\cdot, X',\cdot)$, then $\underline{u} \leq \overline{u}$ in $\mathbb{R} \times (X',+\infty)\times \omega$.
\end{Lemma}
Note that the previous elliptic maximum principle lemma can be seen as a particular case. In fact the proof is quite similar and therefore we only briefly sketch it.
\begin{proof}
We start with the case $c>2\sqrt{g'_+(0)}$, and define $\varepsilon>0$ such that $c>2\sqrt{g'_+(0)+\varepsilon}$. We use some exponential lift as above, so that $\psi=e^{\frac{c}{2}x} (\underline{u} - \overline{u})$ satisfies
$$\partial_t \psi-  \Delta \psi + \frac{c^2}{4} \psi - \frac{ f(x,y,\underline{u}) - f (x,y,\overline{u})}{\underline{u} - \overline{u}} \times \psi   \leq 0.$$
Proceeding as above, there exists $X>0$ such that for all $t>0$, $x\geq X$, $y\in\omega$,
$$\frac{ f(x,y,\overline{u}) - f (x,y,\underline{u})}{\overline{u} - \underline{u}}<g'_+(0)+\frac{\varepsilon}{2}.$$
Now choosing $0<\delta<\varepsilon/2$ and $\gamma>0$, the function $w:=\gamma e^{-\delta t}$ is a supersolution of the above linear parabolic problem for all $t\in \mathbb{R}$, $x\geq X$, $y\in\omega$. 
Noting that $e^{\frac{c}{2} x} \times (\underline{u} - \overline{u})$ is uniformly bounded in space and time, one can choose $\gamma$ so that $\psi < \gamma$ for all time, $x\geq X$ and $y\in\omega$, as well as $\psi (\cdot,X',\cdot) \geq 0$ (for some $X'\geq X$ as assumed in Lemma~\ref{lemma:paramaxprinciple}). By a comparison principle, we get that $\psi (t,\cdot) \leq \gamma e^{-\delta s}$ for all $t \in \mathbb{R}$, $s >0$, all $y \in \omega$ and $x \geq X'$. Therefore $\psi \leq 0 $, and in other words, $\overline{u}- \underline{u} \geq 0$.

The case $c = 2 \sqrt{g'_+ (0)}$ is similar. Posing now $\psi = \frac{e^{\frac{c}{2}x}}{1 + \sqrt{x}} (\underline{u} - \overline{u})$, one may find that $\psi$ is a subsolution of the heat equation on a half-space $(X,+\infty) \times \omega$, for all times. It also satisfies $\psi (\cdot, +\infty, \cdot)$ and the conclusion of Lemma~\ref{lemma:paramaxprinciple} follows.
\end{proof}

\bibliography{biblioBG}
\bibliographystyle{plain}

\end{document}